\documentclass[12pt]{article}
\usepackage{amsmath,amsfonts,amsthm,amssymb,mathrsfs,mathtools}
\usepackage{graphicx}
\usepackage[usenames,dvipsnames]{color}
\usepackage{float}
\usepackage{graphicx}
\usepackage{subfigure}
\usepackage{caption}
\usepackage{cite}
\usepackage{url}
\usepackage{caption}
\usepackage{epstopdf}
\usepackage{hyperref}

\usepackage{color}

\newcommand{\dint}{\displaystyle\int}

\theoremstyle{plain}
\newtheorem{theorem}{Theorem}[section]
\newtheorem{hy}{Assumption}[section]
\newtheorem{corollary}[theorem]{Corollary}
\newtheorem{lemma}[theorem]{Lemma}
\newtheorem{proposition}[theorem]{Proposition}

\theoremstyle{definition}
\newtheorem{definition}[theorem]{Definition}

\theoremstyle{remark}
\newtheorem{remark}[theorem]{Remark}

\numberwithin{equation}{section}
\numberwithin{theorem}{section}

\usepackage{a4wide}
\usepackage{geometry}
\begin{document}

\title{Bridges with random length: \\ Gaussian-Markovian case}

\author{\textbf{Mohamed Erraoui}\\
Universit{\'e} Cadi Ayyad, Facult{\'e} des Sciences Semlalia,\\
 D{\'e}partement de Math{\'e}matiques, B.P. 2390, Marrakech, Maroc\\
Email: erraoui@uca.ma 
\and 
\textbf{Mohammed Louriki}\\
Universit{\'e} Cadi Ayyad, Facult{\'e} des Sciences Semlalia,\\
 D{\'e}partement de Math{\'e}matiques, B.P. 2390, Marrakech, Maroc\\
 Email: louriki.1992@gmail.com}

\maketitle

\begin{abstract}

Motivated by the Brownian bridge on random interval considered by  Bedini et al \cite{BBE}, we introduce and study Gaussian bridges with random length with special emphasis to the Markov property.  We prove that if the starting process is Markov then this property was kept by the bridge with respect to the usual augmentation of its natural filtration.  This leads us to conclude that the completed natural filtration of the bridge satisfies the usual conditions of right-continuity and completeness.  
\end{abstract}
\begin{center}
\section{Introduction}
\label{Setion_1}
\end{center}

A deterministic length bridge is a stochastic process that is pinned to some fixed point at a fixed future time. 
There is extensive literature on this topic dealing with many kinds of bridges, Brownian bridge, generalized Gaussian bridge, Markov bridge, Gamma bridge. We quote \cite{A},\cite{BC}, \cite{GSV}, \cite{SY} and references therein. To our knowledge, the first work that deals with the bridge of random length is that of Bedini et al \cite{BBE}. Precisely, motivated by the problem of modeling the information concerning the default time of a financial company, the authors consider a new approach to credit
risk in which the information about the time of bankruptcy is modelled
using a Brownian bridge $\beta$ that starts at zero and is conditioned
to equal zero when the default occurs. This latter is modeled by a strictly positive random variable $\tau$. The process $\beta$ is called
Brownian bridge of random length (also information process) and is
defined as follows
\begin{equation}
\beta_{t}=W_{t}-\dfrac{t}{\tau\lor t}\,W_{\tau\lor t},\quad t\geq0,
\end{equation}
where $W$ is a Brownian motion independent of $\tau$. It should be mentioned that the process $\beta$ is given by the composition
of the two mappings $(r,t,w)\rightarrow\beta_{t}^{r}(w)$ and $(t,w)\rightarrow(\tau(w),t,\omega)$
where $\beta^{r}$, for a fixed positive real number $r$, is the Brownian bridge of length $r$. It is constructed from 
the orthogonal decomposition of $W$ with respect to $W_{r}$, and has the explicit form
\begin{equation}\label{detergaussbridge}
\beta^{r}_{t}=W_{t}-\dfrac{t}{r\lor t}\,W_{r\lor t},\quad t\geq0.
\end{equation}

In their model, the informations are carried by the process $\beta$
through two filtrations $\mathbb{F}^{P}$ and $\mathbb{F}^{\beta}$. $\mathbb{F}^{P}=\left(\mathcal{F}_{t}^{P}\right)_{t\geq0}$
is the completed natural filtration, that is $\mathcal{F}_{t}^{P}$ is
the observation of the information process $\beta$ up to time $t$
augmented with negligible sets, and $\mathbb{F}^{\beta}=\left(\mathcal{F}_{t}^{\beta}\right)_{t\geq0}$
is the smallest filtration which contains $\mathbb{F}^{P}$ and satisfies
the usual conditions of right-continuity and completeness. It is well known that the stopping time property play a key role in the modeling of default times in mathematical finance. As a first step,
the authors investigate this property for the unknown default time $\tau$ as well as basic properties of the process $\beta$ with respect to the filtration $\mathbb{F}^{P}$. Namely, the Bayes estimates,
the structure of the a posteriori distribution of $\tau$ and the Markov property of  the process $\beta $. 
In the second step, the authors focussed on the Markov property relative to $\mathbb{F}^{\beta}$. This has led to an important fact which is none other than the equality of filtrations $\mathbb{F}^{P}$ and $\mathbb{F}^{\beta}$. In other words only one filtration is used as source of informations. It is interesting to point out that the main tools used are the Gaussian and the Markov properties of the Brownian Motion $W$.
Concerning the question of predictability of $\tau$, it is worth reminding that she was considered in a separate paper by Bedini and Hinz in \cite{BH}. Using a well known fact from parabolic potential theory, they provide a sufficient condition for predictability in terms of the size of the support of its law $P_\tau$.

On the other hand, it is important to underline that in finance the Markov property is one of the most popular assumptions in most continuous-time modeling. For example, in modeling interest rate term structure, such popular models as Vasicek \cite{V}, Cox, Ingersoll, and Ross \cite{CIR} are all Markov processes. Our first object in this article is to give the definition of Gaussian bridges with random length with a degree of generality appropriate, especially the bridges for Markov processes.
To do this, we consider $X=(X_{t})_{t\geq0}$ a continuous centred  Gaussian process with positive
definite covariance function $R_X$. Using the orthogonal decomposition of Hilbert space associated to $X$ with respect to $X_r$, Gasbarra et al \cite{GSV} consider the analog of \eqref{detergaussbridge}, that is a Gaussian bridge with length $r$, for a large class of non-semimartingale processes $X$ as follows
$$  \xi_t^r(\omega)=X_t(\omega)-\dfrac{R_X(t,r\vee t)}{R_X(r\vee t,r\vee t)}X_{r\vee t}(\omega),~ t \geq 0.$$
Inspired by ideas of \cite{BBE}, we will consider similarly the Gaussian bridge of random length associated to the process $X$ by the composition
of the two mappings $(r,t,w)\rightarrow\xi_{t}^{r}(w)$ and $(t,w)\rightarrow(\tau(w),t,\omega)$ which led to
$$  \xi_t:=X_t-\dfrac{R_X(t,\tau \vee t)}{R_X(\tau \vee t,\tau \vee t)}X_{\tau \vee t},~ t \geq 0. $$ 
Firstly, we will deal with the stopping time property of $\tau$ and then we seek additional assumptions under which it becomes predictable. Moreover, we derive, for the process $\xi$, similar properties to those established for the process $\beta$ putting a special emphasis on the Markov property but this requires further study as will be seen below. We also wish to point out the effect that
 the process $\xi$ inherits the Markov property from the process $X$ through $\xi^r$.

Section 2 begins with the definition of Gaussian bridge of random length $\xi $ and some useful 
properties of $\xi^r$ which will be used throughout the paper. In Section 3 we consider the stopping time property 
as well as the question of predictability of $\tau$ with respect to the right continuous and complete filtration generated by the process $\xi$, which will be denoted by $\mathbb{F}^{\xi,c}_+$. Moreover, we give the conditional distribution of $\tau$ and $\xi_u$ given $\xi_{t}$ for $u>t>0$. In section 4 we establish the Markov property of the process $\xi$ with respect to its completed natural filtration.  As a consequence, we derive Bayesian estimates of the distribution of the default time $\tau$ given the past behaviour of the process $\xi$ up to time $t$. Section 5 deals with the Markov property of $\xi$ the Gaussian bridge of random length with respect to $\mathbb{F}^{\xi,c}_+$. \\
For convenience some notations used in the paper are introduced as follows:
For a complete probability space $(\Omega,\mathcal{F},\mathbb{P})$, $\mathcal{N}_p$ denotes the
collection of $\mathbb{P}$-null sets. If $\theta$ is a random variable, then $\mathbb{P}_{\theta}$
denote the law of $\theta$ under $\mathbb{P}$. $C\left(\mathbb{R}_+,\mathbb{R}\right)$ denotes the canonical space that is the space of continuous real-valued functions defined on $\mathbb{R}_+$, $\mathcal{C}$ the $\sigma$-algebra generated by the canonical process. If $E$ is a topological space, then the Borel $\sigma$-algebra over $E$ will be denoted by $\mathcal{B}(E)$. The characteristic function of a set $A$ is written $\mathbb{I}_{A}$. 
The symmetric difference of two sets is denoted by $\Delta$. $p(t, x, y)$, $x\in \mathbb{R}$, denotes the Gaussian density function with variance $t$ and mean $y$. $R_Y(s,t)= \text{cov} (Y_s,Y_t)$, $s,t \in \mathbb{R}_+$ is the covariance function associated to Gaussian process $Y$.
Finally for any process $Y=(Y_t,\, t\geq 0)$ on $(\Omega,\mathcal{F},\mathbb{P})$, we define by:
\begin{enumerate}
	\item[(i)] $\mathbb{F}^{Y}=\bigg(\mathcal{F}^{Y}_t:=\sigma(Y_s, s\leq t),~ t\geq 0\bigg)$ the natural filtration of the process $Y$.
	\item[(ii)] $\mathbb{F}^{Y,c}=\bigg(\mathcal{F}^{Y,c}_t:=\mathcal{F}^{Y}_t\vee \mathcal{N}_{P},\, t\geq 0\bigg)$ the completed natural filtration.
	\item[(iii)] $\mathbb{F}^{Y,c}_{+}=\bigg(\mathcal{F}^{Y,c}_{t^{+}}:=\underset{{s>t}}\bigcap\mathcal{F}^{Y,c}_{s}=\mathcal{F}^{Y}_{t^{+}}\vee \mathcal{N}_{P},\, t\geq 0\bigg)$ the smallest filtration containing $\mathbb{F}^{Y}$ and satisfying the usual hypotheses of right-continuity and completeness.
\end{enumerate}
Throughout this paper we assume the standing assumption 
\begin{hy}\label{hy1}
 $X=(X_t)_{t\geq 0}$ is a centered Gaussian process with continuous sample paths and $X_0=0$ a.s. such that:
	\begin{enumerate}
		\item[(i)]  $R_X(t,t)> 0$ for all $t>0$.
		\item[(ii)] For all $s,t>0$ such that $t\neq s$ we have $$ R_X(t,s)^2<R_X(t,t)R_X(s,s).$$
	\end{enumerate}
\end{hy} 

\begin{remark}
Assumption \eqref{hy1} (ii) means that the covariance matrix of $(X_{s},X_{t})$ is non degenerate, that is $(X_{s},X_{t})$ has a density with respect to the Lebesgue measure on $\mathbb{R}^{2}$.
\end{remark} 

 
\begin{center}
	\section{Definitions and properties}\label{sectiondefandprop}
\end{center}


	The purpose of this section is to define the Gaussian bridge of random length. To achieve this we need to recall the definition of $\xi^r$ the bridge with deterministic length $r>0$. 
	
	\begin{definition}
		Let $r\in (0,+\infty)$. The map $ \xi^r:\Omega \longmapsto C\left(\mathbb{R}_+,\mathbb{R}\right)$ defined by
		$$ \xi_t^r(\omega):=X_t(\omega)-\dfrac{R_X(t,r\vee t)}{R_X(r\vee t,r\vee t)}X_{r\vee t}(\omega),~ t \geq 0,~ \omega \in \Omega,$$
		is called Gaussian bridge of length r associated to $X$.
	\end{definition}
	The next proposition gives some properties of the centred Gaussian process $\xi^r$ for $r>0$.
	\begin{proposition} Let Assumption \eqref{hy1} be satisfied. Then
	\begin{enumerate}
	\item[(i)] $\xi^r_0=0$ a.s.
		\item[(ii)]  For any $0\leq s,t<r$, we have 
\begin{equation}\label{bridgecov}
R_{\xi^r}(s,t)=R_X(s,t) -\dfrac{R_X(s,r)R_X(t,r)}{R_X(r,r)}.
\end{equation}		
		
		\item[(iii)] The probability density  function $\varphi_{\xi_{t}^r}$ is given by
		\begin{equation}
		\varphi_{\xi_{t}^r}(x)= \left\{
		\begin{array}{lll}
		& p\left( \dfrac{A_{t,r}}{R_X(r,r)},x,0\right)  ~\text{if}~ 0<t<r ~, x \in \mathbb{R},  
		
		\\ \\ & 0 ~ \text{if} ~  r \leqslant t  ~, x \in \mathbb{R},
		\end{array}
		\right.\label{density}
		\end{equation}
		where 
		\begin{equation}
			A_{t,s}=R_X(s,s)R_X(t,t)-R_X(s,t)^2, ~\text{for all}~~ s,t\in \mathbb{R}_+.\label{equationdefofA}
		\end{equation}
	\end{enumerate}
	\end{proposition}	
	Let us first give some light on the covariance function of $R_{\xi^r}$. It follows from Assumption \eqref{hy1} (ii) that for any $u, t>0$ with $u\neq t$, the Gaussian vector $(X_{u},X_{t})$ has a non-degenerate covariance matrix. But this is not sufficient in general to ensure that the covariance matrix of $(\xi_{t}^{r},\xi_{u}^{r})$ is invertible for all $0<t,u<r$. However  since 
	\[
\left(\begin{array}{l}
\xi_{u}^{r}\\
\\
\xi_{t}^{r}\\
\\
X_{r}
\end{array}\right)=\left(\begin{array}{lclcl}
1 &  & 0 &  & -\dfrac{R_{X}(u,r)}{R_{X}(r,r)}\\
0 &  & 1 &  & -\dfrac{R_{X}(t,r)}{R_{X}(r,r)}\\
0 &  & 0 &  & 1
\end{array}\right)\left(\begin{array}{l}
X_{u}\\
\\
X_{t}\\
\\
X_{r}
\end{array}\right)
\]
then if we assume that $\left(X_{u},X_{t},X_{r}\right)$ has an absolutely continuous density with respect to the lebesgue
measure on $\mathbb{R}^{3}$ it is simple to see that the Gaussian vector $(\xi_{t}^{r},\xi_{u}^{r})$ has a non-degenerate covariance matrix for all $0<t,u<r$.
 In this case for $t,u \in [0,r)$ such that $t\neq u$, the regular conditional law of
		$\xi_u^r$ given $\xi_t^r$ is given by: 
		\begin{equation}
		\mathbb{P}_{\xi_u^r|\xi_t^r=x}(x,dy)=p\left( \sigma_{t,u}^r,y,\mu_{t,u}^r\,x \right) dy,\label{conditionalnonmarkov}
		\end{equation}
		where 
		$$ \mu_{t,u}^r=\dfrac{R_X(u,t)R_X(r,r)-R_X(u,r)R_X(t,r)}{A_{t,r}},$$
		$$ \sigma_{t,u}^r=\dfrac{A_{u,r}}{R_X(r,r)} -\dfrac{\bigg( R_X(u,t)R_X(r,r)-R_X(u,r)R_X(t,r)\bigg) ^2}{R_X(r,r)A_{t,r} },$$
	 and $x\in \mathbb{R}$.

We now formulate a joint continuity property of the centred Gaussian process $\xi^r$.
	\begin{proposition}\label{propRcontinue}
	
	Under the assumption \eqref{hy1}, the map $(r,t)\longmapsto \xi_t^r(\omega)$ from $(0,+\infty)\times \mathbb{R}_{+}$ into $\mathbb{R}$ is continuous
			for all $\omega \in \Omega$.
		
	\end{proposition}
	\begin{proof}
		Since $X$ has a continuous sample paths then $R_X(\cdot,\cdot)$ is continuous on $(0,+\infty)\times \mathbb{R}_{+}$. Thus for all $\omega \in \Omega$, the map $(r,t)\longmapsto \xi_t^r(\omega)$ is continuous
			 as a composition of continuous maps.
	\end{proof}
	
	It is clear that the process $\xi$ is really a function of the variables $(r,t,\omega)$ and for technical reasons, it is often convenient to have some joint measurability properties.
	\begin{lemma}
Assume that Assumption \eqref{hy1} holds. Then the map $(r,t,\omega)\longmapsto \xi_t^r(\omega)$ of $\big((0,+\infty)\times \mathbb{R}_{+} \times \Omega ,\mathcal{B}\big((0,+\infty)\big)\otimes \mathcal{B}(\mathbb{R}_{+})\otimes\mathcal{F}\big)$ into $(\mathbb{R},\mathcal{B}(\mathbb{R}))$ is measurable. In particular, the t-section of
		$(r,t,\omega)\longmapsto \xi_t^r(\omega)$: $(r,\omega)\longmapsto \xi_t^r(\omega)$ is measurable with respect to the $\sigma$-algebra
		$\mathcal{B}\big((0,+\infty)\big)\otimes\mathcal{F}$, for all $t \geq 0$.
	\end{lemma}
	\begin{proof}
		Since the map $(r,t)\longmapsto \xi_t^r(\omega)$ is continuous for all $\omega \in \Omega$, then the map $(r,t,\omega)\longmapsto \xi_t^r(\omega)$ can be obtained as the pointwise limit of sequences of measurable functions. So, it is sufficient to use standard results on the passage to the limit of sequences of measurable functions.
	\end{proof}
	As a consequence we have the following corollary.
	\begin{corollary}\label{cormesurable}
		Assume that Assumption \eqref{hy1} holds. Then the map $(r,\omega)\longmapsto \xi_t^r(\omega)$ of $\big((0,+\infty)\times \Omega ,\mathcal{B}\big((0,+\infty)\big)\otimes \mathcal{F}\big)$
		into $(C\left(\mathbb{R}_+,\mathbb{R}\right), \mathcal{C})$ is measurable.
	\end{corollary}
	Thanks to the above corollary we could define another process $(\xi_{t},t\geq 0)$ by substituting $r$ by a random time $\tau$. We  now give precise definition. 
	\begin{definition}
	Let $\tau: (\Omega,\mathcal{F},\mathbb{P}) \longmapsto (0,+\infty)$ be a strictly positive random time, with distribution function $F(t) := \mathbb{P}(\tau \leq t)$, $t \geq 0$.
		The map $\xi$:$(\Omega,\mathcal{F})\longmapsto (C\left(\mathbb{R}_+,\mathbb{R}\right),\mathcal{C})$ is defined by 
		$$\xi_{t}(\omega):=\xi_{t}^{r}(\omega)\vert_{r=\tau(\omega)}~~, (t,\omega) \in \mathbb{R}_{+} \times \Omega .$$
		Then $\xi$ takes the form	
		\begin{equation}
		\xi_{t}:=X_{t}-\frac{R_X(t,\tau \vee t)}{R_X(\tau \vee t,\tau \vee t)}X_{\tau \vee t},~~ t\geq 0. \label{defxi}
		\end{equation}
	\end{definition}
	Since $\xi$ is obtained by composition of two maps $(r,t,\omega)\longmapsto \xi^r_t(\omega)$ and $(t,\omega)\longmapsto(\tau(\omega), t, \omega)$, it's not hard to verify that the map $$\xi:(\Omega,\mathcal{F})\longmapsto (C\left(\mathbb{R}_+,\mathbb{R}\right),\mathcal{C})$$ is measurable. The process $\xi$ will be called Gaussian bridge of random length $\tau$.


\begin{center}
	\section{Stopping time property and conditional law}\label{sectionstoppingtimeproperty}
\end{center}


	In this section we prove that the random time $\tau$ is a stopping time
	with respect to $\mathbb{F}^{\xi,c}_{+}$. We give the conditional distribution of the random time $\tau$ given $\xi_{t}$ as well as  the regular conditional law of $\xi_u$ given $\xi_t$. But the latter requires an added assumption on $\xi^r$. At the end we discuss the predictability property with respect to $\mathbb{F}^{\xi,c}_{+}$ under additional assumption on $X$. 
	From now on, we suppose the fundamental assumption.\\
	\begin{hy}\label{hyindependent}
		The random time $\tau$ and the Gaussian process $X$ are independent.
	\end{hy} 
	
	\begin{remark}

It is easy to see that under assumption \eqref{hyindependent}, the conditional law of the process $\xi$ given the random time $\tau$, $\mathbb{P}_{\xi\vert \tau=r}$, is none other than the law of the process $\xi^{r}$. That is, on the canonical space we have
\begin{equation}\label{condlawresptau}
\mathbb{P}_{\xi\vert \tau=r}=\mathbb{P}_{\xi^{r}}.
\end{equation}
	
	\end{remark}

	\begin{proposition}
Assume that Assumptions \eqref{hy1} and \eqref{hyindependent} hold. Then, for all $t>0$, we have $\mathbb{P}\left(\{\xi_t = 0\} \bigtriangleup \{\tau \leq t\}\right)=0$. 
		Then $\tau $ is a stopping time with respect to $\mathbb{F}^{\xi,c}$ and consequently it is a stopping time with respect to $\mathbb{F}^{\xi,c}_{+}$.
	\end{proposition}
	\begin{proof}
		First we have from the definition of $\xi$ that  $\xi_t=0$ for $\tau \leq t$. Then $\{\tau\leq t\}\subseteq \{\xi_t=0\}$. On the other hand, using the formula of total probability and the equality \eqref{condlawresptau}, we obtain
		\begin{align*}
		\mathbb{P}(\xi_{t}=0,t<\tau)&=\dint_{(t,+\infty)} \mathbb{P}(\xi_{t}=0 | \tau=r) \mathbb{P}_{\tau}(dr)
		\\  
		&=\dint_{(t,+\infty)} \mathbb{P}(\xi_{t}^{r}=0) \mathbb{P}_{\tau}(dr) \\
		&=0.
		\end{align*}
	The latter equality uses the fact that $\xi_t^r$ is a Gaussian random variable for $0 < t < r$. 
		Thus $\mathbb{P}\left(\{\xi_t = 0\} \bigtriangleup \{\tau \leq t\}\right)=0$.
		It follows that the event $\{\tau \leq t\}$ belongs to  $\mathcal{F}_t^{\xi} \vee \mathcal{N}_P,$ for all $t \geq 0$. Hence $\tau$ is a stopping time with respect to $\mathbb{F}^{\xi,c}$ and consequently it is also a stopping time with respect to $\mathbb{F}^{\xi,c}_{+}$.
	\end{proof}

In order to determine the conditional law of the random time $\tau$ given $\xi_t$ we will use the following
	\begin{proposition}\label{propbayesestimate}
		Assume that Assumptions \eqref{hy1} and \eqref{hyindependent} hold. Let $t>0$ such that $F(t)>0$ and $g:\mathbb{R}_{+}\longrightarrow \mathbb{R} $ be a Borel function satisfying
		$\mathbb{E}[|g(\tau)|]<+\infty$. Then, $\mathbb{P}$-a.s.,
		\begin{equation}
		\mathbb{E}[g(\tau)|\xi_t]=\dint_{(0,t]}\frac{g(r)}{F(t)}\mathbb{P}_\tau(dr)\mathbb{I}_{\{\xi_t=0 \}}+\dint_{(t,+\infty)}g(r)\phi_{\xi_{t}^r}(\xi_{t})\,\mathbb{P}_{\tau}(dr)\mathbb{I}_{\{\xi_t\neq0 \}}, \label{equationbeyesestimate}
		\end{equation}
		where the function $\phi_{\xi_{t}^r}$ is defined by:
		\begin{equation}
		\phi_{\xi_{t}^r}(x)=\frac{\varphi_{\xi_{t}^r}(x)}{\dint_{(t,+\infty)}\varphi_{\xi_{t}^s}(x)\mathbb{P}_{\tau}(ds)},\quad  x\in \mathbb{R},~~r\in (t,+\infty). \label{equationphi}
		\end{equation}	
	\end{proposition}
	\begin{proof}
		Let us consider the measure $\mu$ defined on $\mathcal{B}(\mathbb{R})$ by
		$$ \mu(dx)=\delta_0(dx)+\lambda (dx),$$ 	
		where $\delta_0(dx)$ and $\lambda (dx)$ are the Dirac measure and the Lebesgue measure on $\mathcal{B}(\mathbb{R})$ respectively. Then for any $B\in \mathcal{B}(\mathbb{R})$ we have   
		$$\mathbb{P}(\xi_{t}\in B|\tau =r)=\mathbb{P}(\xi_{t}^{r}\in B)=\dint_{B}q_{t}(r,x)\mu(dx),$$
		where the function $q_{t}$ is a nonnegative and measurable in the two variables jointly given by $$q_{t}(r,x)=\mathbb{I}_{\{x=0 \}}\mathbb{I}_{\{r\leq t \}}+\varphi_{\xi_{t}^r}(x)\mathbb{I}_{\{x\neq 0 \}}\mathbb{I}_{\{t<r \}}.$$ It follows from Bayes formula (see  \cite{S}  p. 272)
		that $\mathbb{P}$-a.s.:
		\begin{align*}
		\mathbb{E}[g(\tau)|\xi_t]&=\frac{\dint_{(0,+\infty)}g(r)q_{t}(r,\xi_t)\mathbb{P}_{\tau}(dr)}{\dint_{(0,+\infty)}q_{t}(r,\xi_t)\mathbb{P}_{\tau}(dr)}\\
		&=\frac{\dint_{(0,t]}g(r)\mathbb{P}_{\tau}(dr)\mathbb{I}_{\{\xi_t=0 \}}+\dint_{(t,+\infty)}g(r)\varphi_{\xi_{t}^r}(\xi_{t})\mathbb{P}_{\tau}(dr)\mathbb{I}_{\{\xi_t\neq 0 \}}}{F(t)\mathbb{I}_{\{\xi_t=0 \}}+\dint_{(t,+\infty)}\varphi_{\xi_{t}^r}(\xi_{t})\mathbb{P}_{\tau}(dr)\mathbb{I}_{\{\xi_t\neq 0 \}}}\\
		&=\dint_{(0,t]}\frac{g(r)}{F(t)}\mathbb{P}_\tau(dr)\mathbb{I}_{\{\xi_t=0 \}}+\dint_{(t,+\infty)}g(r)\phi_{\xi_{t}^r}(\xi_{t})\mathbb{P}_{\tau}(dr)\mathbb{I}_{\{\xi_t\neq0 \}}.
		\end{align*}
	\end{proof}
\begin{remark}
Since $\mathbb{I}_{\{\xi_t=0\}}=\mathbb{I}_{\{\tau\leq t\}}$ $\mathbb{P}$-a.s we can rewrite formula \eqref{equationbeyesestimate} in the following form
\begin{equation}
		\mathbb{E}[g(\tau)|\xi_t]=\dint_{(0,t]}\frac{g(r)}{F(t)}\mathbb{P}_\tau(dr)\mathbb{I}_{\{\tau\leq t \}}+\dint_{(t,+\infty)}g(r)\phi_{\xi_{t}^r}(\xi_{t})\mathbb{P}_{\tau}(dr)\mathbb{I}_{\{t< \tau \}}. \label{equationbeyesestimaterec}
		\end{equation}

\end{remark}

\begin{corollary}

Assume that Assumptions \eqref{hy1} and \eqref{hyindependent} hold. Then the conditional law of the random time $\tau$ given $\xi_t$ is given by 
\begin{equation}\label{taucondlaw}
\mathbb{P}_{\tau|\xi_t=x}(x,dr)=\frac{1}{F(t)}\mathbb{I}_{(0,t]}(r)\mathbb{P}_\tau(dr)\mathbb{I}_{\{x=0 \}}+\phi_{\xi_{t}^r}(\xi_{t})\mathbb{I}_{(t,+\infty)}(r)\mathbb{P}_{\tau}(dr)\mathbb{I}_{\{x\neq0 \}}.
\end{equation}

\end{corollary}
It is now possible to extend the above proposition by assuming 
			
	\begin{hy}\label{hy2}
		 For any $r>0$, the covariance matrix of $(\xi_{t}^{r},\xi_{u}^{r})$ is invertible for all $0<t,u<r$.		
\end{hy} 
We note that the above assumption is required to ensure the existence of  the regular conditional law of
		$\xi_u^r$ given $\xi_t^r$ given by \eqref{conditionalnonmarkov}.
	\begin{proposition}\label{propextensionxiu}
		Let $0<t<u$ such that $F(t)>0$ and $\mathfrak{g}$ be a bounded measurable function defined on
		$(0,+\infty)\times \mathbb{R}$. Assume that Assumptions \eqref{hy1} and \eqref{hyindependent}-\eqref{hy2} are satisfied. Then, $\mathbb{P}$-a.s., we have 
		
	\[		
\begin{array}{lll}
(i)\qquad \mathbb{E}[\mathfrak{g}(\tau,\xi_{t})|\xi_{t}] & = & \dint_{(0,t]}\frac{\mathfrak{g}(r,0)}{F(t)}\mathbb{P}_{\tau}(dr)\mathbb{I}_{\{\xi_{t}=0\}}+\dint_{(t,+\infty)}\mathfrak{g}(r,\xi_{t})\phi_{\xi_{t}^r}(\xi_{t})\mathbb{P}_{\tau}(dr)\mathbb{I}_{\{\xi_{t}\neq0\}},\\
\\
 & = & \dint_{(0,t]}\frac{\mathfrak{g}(r,0)}{F(t)}\mathbb{P}_{\tau}(dr)\mathbb{I}_{\{\tau\leq t\}}+\dint_{(t,+\infty)}\mathfrak{g}(r,\xi_{t})\phi_{\xi_{t}^r}(\xi_{t})\mathbb{P}_{\tau}(dr)\mathbb{I}_{\{t<\tau\}}.
\end{array}
\]

\[
\begin{array}{lll}
(ii)\qquad \mathbb{E}[\mathfrak{g}(\tau,\xi_{u})|\xi_{t}] & = & \dint_{(0,t]}\frac{\mathfrak{g}(r,0)}{F(t)}\mathbb{P}_{\tau}(dr)\mathbb{I}_{\{\xi_{t}=0\}}+\dint_{(t,u]}\mathfrak{g}(r,0)\phi_{\xi_{t}^r}(\xi_{t})\,\mathbb{P}_{\tau}(dr)\mathbb{I}_{\{\xi_{t}\neq0\}}\\ \\
 &  & +\dint_{(u,+\infty)}\mathfrak{G}_{t,u}(r,\xi_{t})\phi_{\xi_{t}^r}(\xi_{t})\,\mathbb{P}_{\tau}(dr)\mathbb{I}_{\{\xi_{t}\neq0\}}.\\ \\
 & = & \dint_{(0,t]}\frac{\mathfrak{g}(r,0)}{F(t)}\mathbb{P}_{\tau}(dr)\mathbb{I}_{\{\tau\leq t\}}+\dint_{(t,u]}\mathfrak{g}(r,0)\phi_{\xi_{t}^r}(\xi_{t})\,\mathbb{P}_{\tau}(dr)\mathbb{I}_{\{t<\tau\}}\\ \\
 &  & +\dint_{(u,+\infty)}\mathfrak{G}_{t,u}(r,\xi_{t})\phi_{\xi_{t}^r}(\xi_{t})\,\mathbb{P}_{\tau}(dr)\mathbb{I}_{\{t<\tau\}}.
\end{array}
\]
where the function $\mathfrak{G}_{t,u}(r,\cdot)$ is defined on $\mathbb{R}$ via \eqref{conditionalnonmarkov} by 
\begin{eqnarray}
			\mathfrak{G}_{t,u}(r,x)&:=& \mathbb{E}[\mathfrak{g}(r,\xi_{u}^{r})|\xi_{t}^{r}=x] \nonumber\\
			&=& \dint_{\mathbb{R}} \mathfrak{g}(r,y)p(\sigma_{t,u}^r,y,\mu_{t,u}^r\,x)\lambda(dy). \label{Gturx}
			\end{eqnarray}

%

	\end{proposition}

	\begin{proof}
		\begin{enumerate}
			\item[(i)] Using the fact that $\mathbb{P}_{\mathfrak{g}(\tau,\xi_{t})|\xi_t=x}=\mathbb{P}_{\mathfrak{g}(\tau,x)|\xi_t=x}$, the assertion is then  deduced from the Proposition \eqref{propbayesestimate}.
			\item[(ii)] Since $\mathbb{I}_{\{\xi_u=0\}}=\mathbb{I}_{\{\tau\leq u\}}$ $\mathbb{P}$-a.s then we have  $\mathbb{P}$-a.s.

			$$ \mathbb{E}[\mathfrak{g}(\tau,\xi_{u})|\xi_{t}]  = \mathbb{E}[\mathfrak{g}(\tau,0)\mathbb{I}_{\{\tau\leqslant u\}}|\xi_{t}]+\mathbb{E}[\mathfrak{g}(\tau,\xi_{u})\mathbb{I}_{\{u<\tau\}}|\xi_{t}].$$
			From Proposition \eqref{propbayesestimate} we obtain
			$$ \mathbb{E}[\mathfrak{g}(\tau,0)\mathbb{I}_{\{\tau\leq u\}}|\xi_{t} ]= \dint_{(0,t]}\frac{\mathfrak{g}(r,0)}{F(t)}\,\mathbb{P}_\tau(dr)\,\mathbb{I}_{\{\xi_{t}=0\}}+
			\dint_{(t,u]}\mathfrak{g}(r,0)\phi_{\xi_{t}^r}(\xi_{t})\,\mathbb{P}_{\tau}(dr)\,\mathbb{I}_{\{\xi_{t}\neq 0\}}.$$
			Now let us show that		
\begin{equation}\label{secondterm}
\mathbb{E}[\mathfrak{g}(\tau,\xi_{u})\mathbb{I}_{\{u<\tau \}}|\xi_{t}]=\dint_{(u,+\infty)}\mathfrak{G}_{t,u}(r,\xi_{t})\phi_{\xi_{t}^r}(\xi_{t})\,\mathbb{P}_{\tau}(dr)\mathbb{I}_{\{\xi_{t}\neq0\}}.
\end{equation}				
			
			Indeed for any bounded Borel function $h$ we have
			\begin{align*}
			\mathbb{E}[\mathfrak{g}(\tau,\xi_{u})\mathbb{I}_{\{u<\tau \}}h(\xi_{t})]&=\dint_{(u,+\infty)}E[\mathfrak{g}(r,\xi_{u}^{r})h(\xi_{t}^{r})]\mathbb{P}_{\tau}(dr) \\
			&=\dint_{(u,+\infty)}\mathbb{E}[\mathbb{E}[\mathfrak{g}(r,\xi_{u}^{r})|\xi_{t}^{r}]h(\xi_{t}^{r})]\mathbb{P}_{\tau}(dr)  \\
			&=\dint_{(u,+\infty)}\mathbb{E}[\mathfrak{G}_{t,u}(r,\xi_{t}^{r})h(\xi_{t}^{r})]\mathbb{P}_{\tau}(dr).
			\end{align*}
			It follows that 
			\[
\begin{array}{l}
\mathbb{E}[\mathfrak{g}(\tau,\xi_{u})\mathbb{I}_{\{u<\tau\}}h(\xi_{t})] =\mathbb{E}[\mathfrak{G}_{t,u}(\tau,\xi_{t})\mathbb{I}_{\{u<\tau\}}h(\xi_{t})].
\end{array}
\]
		Thus \eqref{secondterm} is obtained from the first case $(i)$.
This ends the proof.
	\end{enumerate}
	
		\end{proof}
		
		\begin{remark}
The results of the Proposition \eqref{propextensionxiu} remain true if we only assume that the function $\mathfrak{g}$ satisfies $\mathbb{E}[|\mathfrak{g}(\tau,\xi_{t})|]<+\infty $ and $\mathbb{E}[|\mathfrak{g}(\tau,\xi_{u})|]<+\infty $.

\end{remark}

	\begin{corollary}\label{corconditionaldistribution}
	Let $0<t<u$ and assume that the assumptions of Proposition \eqref{propextensionxiu} are verified. Then  
\begin{enumerate}
\item[1.] The conditional law of $\xi_u$ given $\xi_t$ is given by:
			$$	\mathbb{P}_{\xi_u|\xi_t=x}(x,dy)=\Bigg[  \mathbb{I}_{\{x=0\}}\mathbb{I}_{\{y=0\}}+\dint_{(t,u]}\phi_{\xi_{t}^r}(x)\mathbb{P}_{\tau}(dr)\mathbb{I}_{\{x\neq0\}}\mathbb{I}_{\{y=0\}}
			$$ $$+\dint_{(u,+\infty)}p(\sigma^r_{t,u},y,\mu_{t,u}^r\,x)\phi_{\xi_{t}^r}(x)\mathbb{P}_{\tau}(dr)\mathbb{I}_{\{x\neq0\}}\mathbb{I}_{\{y\neq0\}}\Bigg]  \mu(dy).$$
\item[2.] For any bounded measurable function $g$ defined on
		$\mathbb{R}$ we have 
	
	\begin{align}
			\mathbb{E}[g(\xi_u)|\xi_{t}]=& g(0) \left( \mathbb{I}_{\{\xi_{t}=0\}} 	+ \dint_{(t,u]}\,\phi_{\xi_{t}^{r}}(\xi_{t})\mathbb{P}_{\tau}(dr)\mathbb{I}_{\{\xi_{t}\neq0\}} \right)\nonumber  \\
	\nonumber \\	& +\dint_{(u,+\infty)}\,K_{t,u}(r,\xi_{t})\phi_{\xi_{t}^{r}}(\xi_{t})\,\mathbb{P}_{\tau}(dr)\,\mathbb{I}_{\{\xi_{t}\neq0\}} .	\label{promarkovonepar}
				\end{align}	
 where the function $K_{t,u}(r,x)$ is defined on $\mathbb{R}$ by 
\begin{eqnarray}
			K_{t,u}(r,x)&:=& \mathbb{E}[g(\xi_{u}^{r})|\xi_{t}^{r}=x] \nonumber\\
			&=& \dint_{\mathbb{R}} g(y)p(\sigma_{t,u}^r,y,\mu_{t,u}^r\,x)\lambda(dy). \label{Hcondlaw}
			\end{eqnarray}	
\end{enumerate}	


	\end{corollary}
	
In the next point we discuss the predictability property of $\tau$. The idea is to see $\tau$ as a hitting time of a
continuous $\mathbb{F}_{+}^{\xi,c}$-adapted process for a closed set. This is obtained under the additional assumption:

\begin{hy}\label{hypred} Let $H\in(0,1)$.  
		 
\begin{enumerate}

\item For any compact interval $I$ in $(0+\infty)$:

\begin{description}

\item[(i)] There exists a positive and finite constant $C_{1}(H,I)$ depending only on $I$ and $H$ such that 
\[
\mathbb{E}[(X_{t}-X_{s})^{2}]\leq C_{1}(H,I)\vert t-s\vert{}^{2H},~\forall s,~t\in I.
\]
\item[(ii)] There exists a constant $C_{2}(H,I)>0$ depending only on $I$ and $H$ such that for all $s,t\in I$,
\[
Var(X_{t}\vert X_{s})\geq C_{2}(H,I)\vert t-s\vert{}^{2H}.
\]
Here $Var(X_{t}\vert X_{s})$ denotes the conditional variance of
$X_{t}$ given $X_{s}$.

\end{description}

\item For any $T>0$, the prediction martingale of $X$ defined by $\hat{X}_{t}^{T}=\mathbb{E}(X_{T}\vert\mathcal{F}_{t^{+}}^{X,c})$,
$t\in[0,T]$, has a strictly increasing bracket $\langle\hat{X}^{T}\rangle$. 

\end{enumerate}
\end{hy} 
Note that the bracket $\langle\hat{X}^{T}\rangle$ is strictly increasing
if and only if the covariance function $R_{\hat{X}^{T}}$ is positive
definite. Indeed, since $\hat{X}^{T}$ is a Gaussian martingale then
we have $R_{\hat{X}^{T}}(t,s)=var(\hat{X}_{t\wedge s}^{T})=\langle\hat{X}^{T}\rangle_{t\wedge s}$.
So, another way of stating Assumption (\ref{hypred}-2) is that the value of $X_{T}$
cannot be predicted for certain by using the information $\mathcal{F}_{t^{+}}^{X,c}$
only.

Let $\rho$ be the metric defined on $\mathbb{R}$ by $\rho(s,t)=\vert t-s\vert^{H}$.
By $B_{\rho}(r)$, we denote an open ball of radius $r$ in the metric
space $\left(\mathbb{R},\rho\right)$. For any $E\subseteq\mathbb{R}$,
the Hausdorff measure, in the metric $\rho$, of $E$ is defined by
\[
\mathcal{H}_{\rho}(E)=\underset{\delta\rightarrow0}{\lim}\inf\left\{ \underset{n\geq1}{\sum}2r_{n}:E\subseteq\underset{n\geq1}{\cup}B_{\rho}(r_{n}),\,r_{n}\leq\delta\right\} .
\]

In order to derive the predictability of $\tau$ the key ingredient
is the following estimate proved in \cite{CX} (Theorem
2.1): Let $X$ satisfying Assumption \eqref{hyindependent} and Assumption (\ref{hypred}-1), $I$ a compact interval in $(0+\infty)$ and $E\subseteq I$ is a Borel set. Then there exists a finite
constant $C(H,I)\geq1$ depending only on $I$ and $H$ such that
\begin{equation}
\mathbb{P}(X(E)\cap\{0\}\neq\varnothing)\leq C(H,I)\mathcal{H}_{\rho}(E),\label{eq:haussd}
\end{equation}
where $X(E)$ denote the range of the set E under $X$.

Let $\Gamma$ denote the support of the law $\mathbb{P}_{\tau}$ .
We now consider the two-dimensional process $Y=(Y_{t},\,t\geq0)$
given by 
\[
Y_{t}=\left[\begin{array}{c}
d(t,\Gamma)\\
\xi_{t}
\end{array}\right]
\]
 where $d(t,\Gamma):=\min\limits _{r\in\Gamma}\vert r-t\vert$. Clearly
the process $Y$ is continuous and $\mathbb{F}_{+}^{\xi}$-adapted,
then its first hitting time of the closed set $\{0\}$, $T_{0}^{Y}:=\inf\left\{ t>0:Y_{t}=0\right\} $
is predictable with respect to $\mathbb{F}_{+}^{\xi}$. On the other
hand, since $Y_{\tau}=0$ , it follows that $T_{0}^{Y}\leq\tau$ $\mathbb{P}$-almost
surely. 
\begin{theorem}
Assume that  $0\notin\Gamma$and $X$ satisfies Assumptions \eqref{hy1}, \eqref{hyindependent} and \eqref{hypred}. If $\mathcal{H}_{\rho}(\Gamma)=0$,
then $\tau$ is a predictable $\mathbb{F}_{+}^{\xi,c}$-stopping time. 
\end{theorem}
\begin{proof}
It is sufficient to verify that $\mathbb{P}(T_{0}^{Y}<\tau)=0$. First
since $\Gamma$ is a closed subset of $(0,+\infty)$ then we have
$T_{0}^{Y}=\inf\left\{ t\in\Gamma:\xi_{t}=0\right\} $. Now taking
in account that, conditionally on $\tau=r$, the process $\xi$ has
the same law as the Gaussian bridge $\xi^{r}$ we obtain
\[
\mathbb{P}\left(T_{0}^{Y}<r\vert\tau=r\right)=\mathbb{P}\left(T_{\Gamma}^{\xi^{r}}<r\right),
\]
where $T_{\Gamma}^{\xi^{r}}:=\inf\left\{ t\in\Gamma:\xi_{t}^{r}=0\right\} $. Hence 
\begin{align*}
\mathbb{P}\left(T_{0}^{Y}<\tau\right) & =\dint_{(0,+\infty)}\mathbb{P}\left(T_{0}^{Y}<r\vert\tau=r\right)\mathbb{P}_{\tau}(dr)\\
 & =\dint_{(0,+\infty)}\mathbb{P}(T_{\Gamma}^{\xi^{r}}<r)\mathbb{P}_{\tau}(dr).
\end{align*}
To complete the proof it suffices to prove that $\mathbb{P}(T_{\Gamma}^{\xi^{r}}<r)=0$
for any $r\in\Gamma$. 

Let $r\in\Gamma$ and $0<S<r$. Since $\Gamma$ is a closed and $0\notin\Gamma$
then $$\left\{ T_{\Gamma}^{\xi^{r}}\leq S\right\} =\left\{ \xi^{r}\left(\Gamma\cap[0,S]\right)\cap\{0\}\neq\varnothing\right\}. $$
Moreover it is known that under Assumption (\ref{hypred}-2) the laws of the processes
$\xi^{r}$ and $X$ are equivalent on $[0,S]$, see \cite{GSV}.
Then there exists a probability measure $\mathbb{Q}^{S}$ equivalent
to $\mathbb{P}$ such that the law of $\left\{ \xi_{t}^{r},t\in[0,S]\right\} $
under $\mathbb{Q}^{S}$ is the same as the law of $\left\{ X_{t},t\in[0,S]\right\} $
under $\mathbb{P}$. It follows that 
\[
\mathbb{Q}\left\{ T_{\Gamma}^{\xi^{r}}\leq S\right\} =\mathbb{Q}\left\{ \xi^{r}\left(\Gamma\cap[0,S]\right)\cap\{0\}\neq\varnothing\right\} =\mathbb{P}(X(\Gamma\cap[0,S])\cap\{0\}\neq\varnothing)=0,
\]
where the last equality is a consequence of $\mathcal{H}_{\rho}(\Gamma)=0$
and estimation (\ref{eq:haussd}). Since $\mathbb{Q}^{S}$ and $\mathbb{P}$
are equivalent we obtain $\mathbb{P}\left\{ T_{\Gamma}^{\xi^{r}}\leq S\right\} =0$
for all $S<r$. Consequently we obtain $\mathbb{P}(T_{\Gamma}^{\xi^{r}}<r)=0$.
\end{proof}


	\begin{center}
		\section{Markov property and bayes estimate of the default time $\tau$}\label{sectionmarkovproperty}
	\end{center} 


The main goal, in this section, being to prove	the Markov property of the process $\xi$ with
	respect to its natural filtration $\mathbb{F}^{\xi}$. In order to reach this objective, some preliminary results have been established. We begin by proving the  Markov property of the Gaussian bridge $\xi^r$with respect to $\mathbb{F}^{\xi^r}$. Afterwards we derive two formulas giving the Bayes estimate of $\tau $ as well as the prediction of the information process $\xi$ at some time $u$ given  the $n$-coordinate of $\xi$. Once this has been done, the Markov property of $\xi$ follows quite easily. 
	 From now on we assume the following assumption:\\
	\begin{hy}\label{hymarkov}
	The Gaussian process $X=(X_s)_{s\geq0}$ is a Markov process with respect to its natural filtration $\mathbb{F}^{X}$. 
	\end{hy}
	\begin{remark}
It is convenient to recall the well-known fact: for a centred Gaussian process $Y$, the Markov property with respect to the natural filtration $\mathbb{F}^{Y}$ is characterized by 
	\begin{equation}\label{markovcaract}
		 R_Y(s,t)R_Y(t,u)=R_Y(t,t)R_Y(s,u),
	\end{equation}
for every $s<t<u$, 	see \cite{RY} exercise 1.13, p.86. Hence the covariance function $R_X$ satisfies \eqref{markovcaract}. Moreover it is a continuous strictly definite positive function on $\mathbb{R}_{+}^{*}\times \mathbb{R}_{+}^{*}$. Then Lemma 5.1.9, see \cite{MR}; p.201 (see also exercise 1.13 in \cite{RY}), tells us that $R_X$ can be expressed as:
	\begin{equation}\label{covrep}
		R_X(s,t)=\rho(\inf(s,t))\,q(\sup(s,t)),\; \text{for all} \; s,t\in \mathbb{R}_{+}^{*}.
	\end{equation}
	Here $\rho$ and $q$ are continuous strictly positive functions on $\mathbb{R}_{+}^{*}$ such that $\rho/q$ is non decreasing.
	Since $X=(X_t)_{t\geq 0}$ is a countinuous such that $X_0=0$ then we extend the functions $\rho$ and $q$ to
	$\mathbb{R}_{+}$ by setting $\rho(0)=0$ and $q(0)$ is any positive constant. Thus the representation \eqref{covrep} remains true on $\mathbb{R}_{+}\times \mathbb{R}_{+}$.
	\end{remark}
We now come to the one of the main results of this section, namely the Markov property of the Gaussian bridge $\xi^r$ with deterministic length $r$. A convenient way to prove this is to show that its covariance function $R_{\xi^r}$ satisfies \eqref{covrep}. This is the first claim of the following proposition, after which we turn our attention to a few expressions for regular conditional distribution, probability density function and finite dimensional distribution useful in handling Markov property of $\xi$.
	\begin{proposition}\label{propximarkovandcovariance}
Assume that Assumptions \eqref{hy1}, \eqref{hyindependent} and \eqref{hymarkov} hold.  Then, for any $r>0$, we have
		\begin{enumerate}
			\item[(i)] $\xi^r$ is a Markov process with respect to its completed natural filtration with the covariance function $R_{\xi^{r}}$ given by \begin{equation}\label{covrepdetbri}
		R_{\xi^{r}}(s,t)=\rho(\inf(s,t))\,\tilde{q} (\sup(s,t)),\; \text{for all} \; s,t\in [0,r],
	\end{equation}
	where $$\tilde{q} (u)=q(u)-\rho(u)\dfrac{q(r)}{\rho(r)},\, u\in [0,r].$$
			\item[(ii)]  For $t<u\in [0, r)$,  Assumption \eqref{hy2} holds. Moreover, the regular conditional distribution of
			$\xi_u^r$, given $\xi_t^r$ expressed in formula \eqref{conditionalnonmarkov} can be rewritten, for all $x\in \mathbb{R}$, as: 
			
			\begin{equation}
			\mathbb{P}_{u,t}(x,dy):=\mathbb{P}_{\xi_u^r|\xi_t^r=x}(x,dy) =p\left( \dfrac{B_{t,u}B_{u,r}}{B_{t,r}},y,\dfrac{B_{u,r}}{B_{t,r}}x\right) \lambda(dy),\label{conditionalmarkov}
			\end{equation}
			where 
			\begin{equation}\label{Bdef}
				B_{s,t}=\rho(\sup(s,t))\,q(\inf(s,t))-\rho(\inf(s,t))\,q(\sup(s,t)),
			\end{equation}
			 for all $\; s,t\in [0,r]$.
			\item[(iii)]
Since, for $0<t<r$, we have $\dfrac{A_{t,r}}{R_{X}(r,r)}=\dfrac{\rho(t)B_{t,r}}{\rho(r)}$
then the probability density function $\varphi_{\xi_{t}^{r}}$ is
transformed into the following form

\begin{eqnarray}
\varphi_{\xi_{t}^{r}}(x)& =&  p\left(\dfrac{\rho(t)\,B_{t,r}}{\rho(r)},x,0\right) \nonumber\\
\nonumber \\
   & = &\dfrac{1}{\sqrt{2\pi}}\left(\dfrac{\rho(r)}{\rho(t)\,B_{t,r}}\right)^{\frac{1}{2}}\exp\left(-\dfrac{1}{2}\dfrac{\rho(r)}{\rho(t)\,B_{t,r}}\,x^{2}\right),\ x\in\mathbb{R}.\label{phimarkdensity}
\end{eqnarray}

\item[(iv)] Using the fact that 
\[
\dfrac{\rho(r)}{\rho(t)\,B_{t,r}}-\dfrac{\rho(s)}{\rho(t)\,B_{t,s}}=\dfrac{\rho(s)q(r)-\rho(r)q(s)}{B_{t,r}B_{t,s}}
\]
 the function $\phi_{\xi_{t}^{r}}$ has the following expression

\begin{equation}\label{phinewexp}
\phi_{\xi_{t}^{r}}(x)  = \left(\dint_{(t,+\infty)}\left(\dfrac{\rho(s)B_{t,r}}{\rho(r)B_{t,s}}\right)^{\frac{1}{2}}\exp\left[\dfrac{1}{2}\,\left(\dfrac{\rho(s)q(r)-\rho(r)q(s)}{B_{t,r}B_{t,s}}\right)x^{2}\right]\mathbb{P}_{\tau}(ds)\right)^{-1},
\end{equation}
 for $x\in \mathbb{R}$ and $r\in (t,+\infty)$.

		\item[(v)] If  $0< t_1 <\ldots< t_n $ and $B_0,B_1,\ldots,B_n \in \mathcal{B}(\mathbb{R})$ then
		
		\begin{align}\label{bridgemultlaw}
\mathbb{P}\left(\xi_{0}^{r}\in B_{0},\xi_{t_{1}}^{r}\in B_{1},\ldots,\xi_{t_{n}}^{r}\in B_{n}\right) =\hspace{4cm} & \nonumber \\
\dint_{B_{0}}\delta_{0}(dx_{0})\dint_{B_{1}}\mathbb{P}_{0,t_{1}}(x_{0},dx_{1})\ldots\dint_{B_{m}}\mathbb{P}_{t_{m-1},t_{m}}(x_{m-1},dx_{m})\times\overset{n}{\underset{j=m+1}{\prod}}\delta_{0}(B_{j}), &
\end{align}

where $m=\sup\left\{ k\in\left\{ 1,\ldots,n\right\} :t_{k}<r\right\} $.

				\end{enumerate}
	\end{proposition}
	
	It should be noted that, for all $n\in \mathbb{N}^*$ and $0< t_1 <\ldots< t_n<r$, the Gaussian vector $\left(\xi_{t_{1}}^{r},\ldots,\xi_{t_{n}}^{r}\right)$  has an absolutely continuous density with respect to the Lebesgue measure on $\mathbb{R}^{n}$. From now on, it will be noted by $\varphi_{\xi_{t_{1}}^{r},...,\xi_{t_{n}}^{r}}$ and his expression will be given in  the following

\begin{lemma}\label{lmtoprovemarkov}
		Let $n$ be an integer greater than $1$  and $0<t_1<t_2<...<t_n<r$. Then
		\begin{align}
			\varphi_{\xi_{t_{1}}^{r},...,\xi_{t_{n}}^{r}}(x_1,x_2,...,x_n)&=(2\pi)^{-\frac{n}{2}}\left( \dfrac{\rho(r)}{\rho(t_1)B_{t_n,r}}\prod\limits_{k=2}^{n}\dfrac{1}{B_{t_{k-1},t_k}}\right)^{\frac{1}{2}} \nonumber\\
		\nonumber \\	& \times \exp \left[ -\frac{1}{2}\,\dfrac{\rho(t_2)}{\rho(t_1)B_{t_1,t_2}}\,x_1^2-\frac{1}{2}\sum\limits_{k=2}^{n-1}\dfrac{B_{t_{k-1},t_{k+1}}}{B_{t_{k-1},t_{k}}B_{t_k,t_{k+1}}} \,x_k^2\right.  \nonumber \\
		\nonumber \\	& \left.+\sum\limits_{k=1}^{n-1}\dfrac{x_kx_{k+1}}{B_{t_k,t_{k+1}}}-\frac{1}{2}\dfrac{B_{t_{n-1},r}}{B_{t_n,r}B_{t_{n-1},t_n}} \, x_n^2\right].\label{equationvarphiforn}
		\end{align}

	\end{lemma}
	\begin{proof}
		Since $\xi^r$ is a Gaussian Markov process, then using \eqref{conditionalmarkov} and \eqref{bridgemultlaw} we obtain
		
		$$\varphi_{\xi_{t_{1}}^{r},...,\xi_{t_{n}}^{r}}(x_1,x_2,...,x_n)=\varphi_{\xi_{t_{1}}^{r}}(x_1)\prod\limits_{k=2}^{n}p\bigg(\dfrac{B_{t_{k-1},t_k}B_{t_k,r}}{B_{t_{k-1},r}},x_k,\dfrac{B_{t_k,r}}{B_{t_{k-1},r}}x_{k-1}\bigg).$$ 
Using the following expression 
\[
\begin{array}{lll}
p\left( \dfrac{B_{t_{k-1},t_{k}}B_{t_{k},r}}{B_{t_{k-1},r}},x_{k},\dfrac{B_{t_{k},r}}{B_{t_{k-1},r}}x_{k-1} \right) & = & (2\pi)^{-\frac{1}{2}}\Bigg(\dfrac{B_{t_{k-1},r}}{B_{t_{k-1}t_{k}}B_{t_{k},r}}\Bigg)^{\frac{1}{2}}\\
\\
 &  & \times\exp\left[-\dfrac{1}{2}\dfrac{B_{t_{k-1},r}}{B_{t_{k-1},t_{k}}B_{t_{k},r}}\, x_{k}^{2}+\dfrac{x_{k-1}x_{k}}{B_{t_{k-1},t_{k}}}-\dfrac{1}{2}\dfrac{B_{t_{k},r}}{B_{t_{k-1},t_{k}}B_{t_{k-1},r}}\, x_{k-1}^{2}\right]
\end{array}
\]
yields
\[
\begin{array}{ll}
\prod\limits _{k=2}^{n}p\left(\dfrac{B_{t_{k-1},t_{k}}B_{t_{k},r}}{B_{t_{k-1},r}},x_{k},\dfrac{B_{t_{k},r}}{B_{t_{k-1},r}}x_{k-1}\right) & =(2\pi)^{-\frac{n-1}{2}}\prod\limits _{k=2}^{n}\left(\dfrac{B_{t_{k-1},r}}{B_{t_{k-1}t_{k}}B_{t_{k},r}}\right)^{\frac{1}{2}}\\
\\
 & \times\exp\left[-\dfrac{1}{2}\,\dfrac{B_{t_{2},r}}{B_{t_{1},t_{2}}B_{t_{1},r}}\, x_{1}^{2}-\dfrac{1}{2}\sum\limits_{k=2}^{n-1}\left(\dfrac{B_{t_{k-1},r}}{B_{t_{k-1},t_{k}}B_{t_{k},r}}+\dfrac{B_{t_{k+1},r}}{B_{t_{k},t_{k+1}}B_{t_{k},r}}\right)\,x_{k}^{2}\right.\\
\\
 & \left.+\sum\limits_{k=2}^{n-1}\,\dfrac{x_{k-1}x_{k}}{B_{t_{k-1},t_{k}}}-\dfrac{1}{2}\,\dfrac{B_{t_{n-1},r}}{B_{t_{n-1},t_{n}}B_{t_{n},r}}\,x_{n}^{2}\right]
\end{array}
\]
Hence
\[
\begin{array}{ll}
\varphi_{\xi_{t_{1}}^{r},...,\xi_{t_{n}}^{r}}(x_{1},x_{2},\ldots,x_{n}) & =(2\pi)^{-\frac{n}{2}}\left(\dfrac{\rho(r)}{\rho(t_1)B_{t_{1},r}}\right)^{\frac{1}{2}}\prod\limits _{k=2}^{n}\left(\dfrac{B_{t_{k-1},r}}{B_{t_{k-1}t_{k}}B_{t_{k},r}}\right)^{\frac{1}{2}}\\
\\
  &\times\exp\left[-\dfrac{1}{2}\,\left(\dfrac{\rho(r)}{\rho(t_{1})B_{t_{1},r}}+\dfrac{B_{t_{2},r}}{B_{t_{1},t_{2}}B_{t_{1},r}}\right)\,x_{1}^{2}\right. 
  \\
  \\ & -\left. \dfrac{1}{2}\sum\limits_{k=2}^{n-1}\left(\dfrac{B_{t_{k-1},r}}{B_{t_{k-1},t_{k}}B_{t_{k},r}}+\dfrac{B_{t_{k+1},r}}{B_{t_{k},t_{k+1}}B_{t_{k},r}}\right)\,x_{k}^{2}\right.\\
\\
 & \left.+\sum\limits_{k=2}^{n-1}\,\dfrac{x_{k-1}x_{k}}{B_{t_{k-1},t_{k}}}-\dfrac{1}{2}\,\dfrac{B_{t_{n-1},r}}{B_{t_{n-1},t_{n}}B_{t_{n},r}} \,x_{n}^{2}\right].
\end{array}
\]
Now using the fact that 
\[
\dfrac{\rho(r)}{\rho(t_{1})B_{t_{1},r}}+\dfrac{B_{t_{2},r}}{B_{t_{1},t_{2}}B_{t_{1},r}}=\dfrac{\rho(t_{2})}{\rho(t_{1})B_{t_{1},t_{2}}}
\]
and 
\[
\dfrac{B_{t_{k-1},r}}{B_{t_{k-1},t_{k}}B_{t_{k},r}}+\dfrac{B_{t_{k+1},r}}{B_{t_{k},t_{k+1}}B_{t_{k},r}}=\dfrac{B_{t_{k-1},t_{k+1}}}{B_{t_{k-1},t_{k}}B_{t_{k},t_{k+1}}},\;\;  k=2,3,\ldots,n-1,
\]
we arrive at \eqref{equationvarphiforn}.
\end{proof}	
In order to be able to state the main result of this section we will need to extend the Propositions \eqref{propbayesestimate} and \eqref{propextensionxiu}. More precisely, for $n\in \mathbb{N}^*$, we give the regular conditional distribution of $\tau$ and $(\tau, \xi_\cdot)$ given  the $n$-coordinate of $\xi$.
	\begin{proposition}\label{propbayesestimatejusquatn}
Assume that Assumptions \eqref{hy1}, \eqref{hyindependent} and \eqref{hymarkov} hold. Let $n\in \mathbb{N}^*$ and $0=t_0<t_1<t_2<...<t_n$ such that $F(t_1)>0$.  Let $g:\mathbb{R}_{+}\longrightarrow \mathbb{R} $ be a Borel function satisfying $\mathbb{E}[|g(\tau)|]<+\infty$. Then, $\mathbb{P}$-a.s., we have

\begin{align}\label{taucondxi}
\mathbb{E}[g(\tau)\vert\xi_{t_{1}},\ldots,\xi_{t_{n}}]= & \dint_{(0,t_{1}]}\frac{g(r)}{F(t_{1})}\mathbb{P}_{\tau}(dr)\;\mathbb{I}_{\{\xi_{t_{1}}=0\}} \nonumber \\ \nonumber
\\
 & +\sum\limits _{k=1}^{n-1}\dint_{(t_{k},t_{k+1}]}g(r)\psi_{k}(r,\xi_{t_{k}}) 
 \mathbb{P}_{\tau}(dr)\; \mathbb{I}_{\{\xi_{t_{k}}\neq0,\xi_{t_{k+1}}=0\}} \nonumber  \\ \nonumber
\\
 & +\dint_{(t_{n},+\infty)}g(r)\phi_{\xi_{t_{n}}^{r}}(\xi_{t_{n}})\mathbb{P}_{\tau}(dr)\; \mathbb{I}_{\{\xi_{t_{n}}\neq0\}}
\end{align}
		where the functions $\psi_{k}(r,x)$ are defined on $\mathbb{R}$ by:
		
\begin{equation}
\psi_{1}(r,x):=\dfrac{\varphi_{\xi_{t_{1}}^{r}}(x)}{\dint_{(t_{1},t_{2}]}\varphi_{\xi_{t_{1}}^{s}}(x)\mathbb{P}_{\tau}(ds)}=\dfrac{\left(\dfrac{\rho(r)}{B_{t_{1},r}}\right)^{\frac{1}{2}}\exp\left(-\dfrac{1}{2}\dfrac{\rho(r)}{\rho(t_{1})\,B_{t_{1},r}}\,x^{2}\right)}{\dint_{(t_{1},t_{2}]}\left(\dfrac{\rho(s)}{B_{t_{1},s}}\right)^{\frac{1}{2}}\exp\left(-\dfrac{1}{2}\dfrac{\rho(s)}{\rho(t_{1})\,B_{t_{1},s}}\,x^{2}\right)\mathbb{P}_{\tau}(ds)}
\end{equation}	
	and 	\begin{equation}
		\psi_{k}(r,x):=\dfrac{\left(\dfrac{\rho(r)}{B_{t_{k},r}}\right)^{\frac{1}{2}}\exp\left(-\dfrac{1}{2}\,\dfrac{B_{t_{k-1},r}}{B_{t_{k},r}B_{t_{k-1},t_{k}}}\, x^{2}\right)}{\dint_{(t_{k},t_{k+1}]}\left(\dfrac{\rho(s)}{B_{t_{k},s}}\right)^{\frac{1}{2}}\exp\left(-\frac{1}{2}\,\dfrac{B_{t_{k-1},s}}{B_{t_{k},s}B_{t_{k-1},t_{k}}}\, x^{2} \right)\mathbb{P}_{\tau}(ds)},\; k=2,\ldots,n-1.
		\end{equation}	
%
%
		
	\end{proposition}
	\begin{proof} It follows from the equality in law \eqref{condlawresptau} that, for $B_1,B_2,\ldots,B_n \in \mathcal{B}(\mathbb{R})$, we have 

\[
	\begin{array}{c}
\mathbb{P}\left((\xi_{t_1},\ldots,\xi_{t_n})\in B_1\times \ldots \times B_n\vert \tau=r\right)=\mathbb{P}\left((\xi_{t_{1}}^{r},\ldots\xi_{t_{n}}^{r})\in B_{1}\times\ldots\times B_{n}\right)= \\ \\\overset{n}{\underset{k=1}{\prod}}\delta_{0}(B_{k})\mathbb{I}_{\{r\leq t_{1}\}}+

\sum\limits _{k=1}^{n-1}\mathbb{P}_{\xi_{t_{1}}^{r},...,\xi_{t_{k}}^{r}}\left(B_{1}\times\ldots\times B_{k}\right)\overset{n}{\underset{j=k+1}{\prod}}\delta_{0}(B_{j})\mathbb{I}_{\{t_{k}<r\leq t_{k+1}\}}

+\mathbb{P}_{\xi_{t_{1}}^{r},\ldots,\xi_{t_{n}}^{r}}\left(B_{1}\times\ldots\times B_{n}\right)\mathbb{I}_{\{t_{n}<r\}}
\\ \\ =\dint_{B_{1}\times\ldots\times B_{n}}q_{t_{1},\ldots,t_{n}}(r,x_{1},\ldots,x_{n})\mu(dx_{1},\ldots,dx_{n}).

\end{array}
\]
Here the function $q_{t_{1},\ldots,t_{n}}$ is a nonnegative and measurable in the $(n+1)$ variables jointly given by 
\[
\begin{array}{ll}
q_{t_{1},\ldots,t_{n}}(r,x_{1},\ldots,x_{n})= & \overset{n}{\underset{k=1}{\prod}}\mathbb{I}_{\{0\}}(x_{k})\; \mathbb{I}_{\{r\leq t_{1}\}}\\
 & +\sum\limits _{k=1}^{n-1}\varphi_{\xi_{t_{1}}^{r},\ldots,\xi_{t_{k}}^{r}}\left(x_{1},\ldots,x_{k}\right)\; \overset{k}{\underset{j=1}{\prod}}\mathbb{I}_{\mathbb{R}^*}(x_{j})\; \overset{n}{\underset{j=k+1}{\prod}}\mathbb{I}_{\{0\}}(x_{j})\; \mathbb{I}_{\{t_{k}<r\leq t_{k+1}\}}\\
 & +\varphi_{\xi_{t_{1}}^{r},\ldots,\xi_{t_{n}}^{r}}\left(x_{1},\ldots,x_{n}\right)\; \overset{n}{\underset{j=1}{\prod}}\mathbb{I}_{\mathbb{R}^*}(x_{j})\mathbb{I}_{\{t_{n}<r\}}
\end{array}
\]
and $\mu(dx_{1},dx_{2},\ldots,dx_{n})$ is a $\sigma$-finite measure on $\mathbb{R}^n$ given by
\[
\begin{array}{ll}
\mu(dx_{1},dx_{2},\ldots,dx_{n})= &  \overset{n}{\underset{k=1}{\bigotimes}}\delta_{0}(dx_{k})\\
 & +\sum\limits _{k=1}^{n-1}\lambda\left(dx_{1},\ldots,dx_{k}\right)\bigotimes\overset{n}{\underset{j=k+1}{\bigotimes}}\delta_{0}(B_{j})\\
 & +\lambda\left(dx_{1},\ldots,dx_{k}\right).
\end{array}
\]
We get from Bayes formula  
\[
\mathbb{E}[g(\tau)\vert\xi_{t_{1}},\ldots,\xi_{t_{n}}]=\dfrac{\dint_{(0,+\infty )}g(r)q_{t_{1},\ldots,t_{n}}(r,\xi_{t_{1}},\ldots,\xi_{t_{n}})\mathbb{P}_{\tau}(dr)}{\dint_{(0,+\infty )}q_{t_{1},\ldots,t_{n}}(r,\xi_{t_{1}},\ldots,\xi_{t_{n}})\mathbb{P}_{\tau}(dr)}.
\]
By a simple integration we find
\[
\begin{array}{l}
\dint_{(0,+\infty)}g(r)q_{t_{1},\ldots,t_{n}}(r,\xi_{t_{1}},\ldots,\xi_{t_{n}})\mathbb{P}_{\tau}(dr)  =\dint_{(0,t_{1}]}g(r)\mathbb{P}_{\tau}(dr)\mathbb{I}_{\{\xi_{t_{1}}=0,\ldots,\xi_{t_{n}}=0\}}\\
\\
  +\sum\limits _{k=1}^{n-1}\dint_{(t_{k},t_{k+1}]}g(r)\varphi_{\xi_{t_{1}}^{r},\ldots,\xi_{t_{k}}^{r}}\left(\xi_{t_{1}},\ldots,\xi_{t_{k}}\right)\mathbb{P}_{\tau}(dr)\mathbb{I}_{\{\xi_{t_{1}}\neq0,\ldots,\xi_{t_{k}}\neq0,\xi_{t_{k+1}}=0,\ldots,\xi_{t_{n}}=0\}}\\
\\
  +\dint_{(t_{n},+\infty)}g(r)\varphi_{\xi_{t_{1}}^{r},\ldots,\xi_{t_{n}}^{r}}\left(\xi_{t_{1}},\ldots,\xi_{t_{n}}\right)\mathbb{P}_{\tau}(dr)\mathbb{I}_{\{\xi_{t_{1}}\neq0,\ldots,\xi_{t_{n}}\neq0\}},
\end{array}
\]
and 
\[
\begin{array}{l}
\dint_{(0,+\infty)}q(r,\xi_{t_{1}},\ldots,\xi_{t_{n}})\mathbb{P}_{\tau}(dr)  =F(t_{1})\mathbb{I}_{\{\xi_{t_{1}}=0,\ldots,\xi_{t_{n}}=0\}}\\
\\
  +\sum\limits _{k=1}^{n-1}\dint_{(t_{k},t_{k+1}]}\varphi_{\xi_{t_{1}}^{r},\ldots,\xi_{t_{k}}^{r}}\left(\xi_{t_{1}},\ldots,\xi_{t_{k}}\right)\mathbb{P}_{\tau}(dr)\mathbb{I}_{\{\xi_{t_{1}}\neq0,\ldots,\xi_{t_{k}}\neq0,\xi_{t_{k+1}}=0,\ldots,\xi_{t_{n}}=0\}}\\
\\
  +\dint_{(t_{n},+\infty)}\varphi_{\xi_{t_{1}}^{r},\ldots,\xi_{t_{n}}^{r}}\left(\xi_{t_{1}},\ldots,\xi_{t_{n}}\right)\mathbb{P}_{\tau}(dr)\mathbb{I}_{\{\xi_{t_{1}}\neq0,\ldots,\xi_{t_{n}}\neq0\}}.
\end{array}
\]
Since $\{\xi_{t_{i}}=0\}\subset\{\xi_{t_{j}}=0\}$ if $j\geq i$
and therefore $\{\xi_{t_{j}}\neq0\}\subset\{\xi_{t_{i}}\neq0\}$ we
obtain
\[
\begin{array}{ll}
\mathbb{E}[g(\tau)\vert\xi_{t_{1}},\ldots,\xi_{t_{n}}]= & \dint_{(0,t_{1}]}\frac{g(r)}{F(t_{1})}\mathbb{P}_{\tau}(dr)\mathbb{I}_{\{\xi_{t_{1}}=0\}}\\
\\
 & +\sum\limits _{k=1}^{n-1}\dint_{(t_{k},t_{k+1}]}g(r)\dfrac{\varphi_{\xi_{t_{1}}^{r},\ldots,\xi_{t_{k}}^{r}}\left(\xi_{t_{1}},\ldots,\xi_{t_{k}}\right)}{\dint_{(t_{k},t_{k+1}]}\varphi_{\xi_{t_{1}}^{s},...,\xi_{t_{k}}^{s}}\left(\xi_{t_{1}},\ldots,\xi_{t_{k}}\right)\mathbb{P}_{\tau}(ds)}\mathbb{P}_{\tau}(dr)\mathbb{I}_{\{\xi_{t_{k}}\neq0,\xi_{t_{k+1}}=0\}}\\
\\
 & +\dint_{(t_{n},+\infty)}g(r)\dfrac{\varphi_{\xi_{t_{1}}^{r},\ldots,\xi_{t_{n}}^{r}}\left(\xi_{t_{1}},\ldots,\xi_{t_{n}}\right)}{\dint_{(t_{n},+\infty)}\varphi_{\xi_{t_{1}}^{s},\ldots,\xi_{t_{n}}^{s}}\left(\xi_{t_{1}},\ldots,\xi_{t_{n}}\right)\mathbb{P}_{\tau}(ds)}\mathbb{I}_{\{\xi_{t_{n}}\neq0\}}.
\end{array}
\]
Now for $k=1$ we use \eqref{phimarkdensity} to see that
\[
\dfrac{\varphi_{\xi_{t_{1}}^{r}}(x)}{\dint_{(t_{1},t_{2}]}\varphi_{\xi_{t_{1}}^{s}}(x)\mathbb{P}_{\tau}(ds)}=\dfrac{\left(\dfrac{\rho(r)}{B_{t_{1},r}}\right)^{\frac{1}{2}}\exp\left(-\dfrac{1}{2}\dfrac{\rho(r)}{\rho(t_{1})\,B_{t_{1},r}}\,x^{2}\right)}{\dint_{(t_{1},t_{2}]}\left(\dfrac{\rho(s)}{B_{t_{1},s}}\right)^{\frac{1}{2}}\exp\left(-\dfrac{1}{2}\dfrac{\rho(s)}{\rho(t_{1})\,B_{t_{1},s}}\,x^{2}\right)\mathbb{P}_{\tau}(ds)}.
\]
On the oder hand for $k=2,\ldots,n-1$, using \eqref{equationvarphiforn} it follows that
\[
\begin{array}{l}
\dfrac{\varphi_{\xi_{t_{1}}^{r},\ldots,\xi_{t_{k}}^{r}}\left(x_{1},\ldots,x_{k}\right)}{\dint_{(t_{k},t_{k+1}]}\varphi_{\xi_{t_{1}}^{s},\ldots,\xi_{t_{k}}^{s}}\left(x_{1},\ldots,x_{k}\right)\mathbb{P}_{\tau}(ds)}=  \dfrac{\left(\dfrac{\rho(r)}{B_{t_{k},r}}\right)^{\frac{1}{2}}\exp\left(-\dfrac{1}{2}x_{k}^{2}\dfrac{B_{t_{k-1},r}}{B_{t_{k},r}B_{t_{k-1},t_{k}}}\right)}{\dint_{(t_{k},t_{k+1}]}\left(\dfrac{\rho(s)}{B_{t_{k},s}}\right)^{\frac{1}{2}}\exp\left(-\frac{1}{2}x_{k}^{2}\dfrac{B_{t_{k-1},s}}{B_{t_{k},s}B_{t_{k-1},t_{k}}}\right)\mathbb{P}_{\tau}(ds)}
\end{array}
\]
Finally we obtain from \eqref{phinewexp} that
\[
\begin{array}{l}
\dfrac{\varphi_{\xi_{t_{1}}^{r},\ldots,\xi_{t_{n}}^{r}}\left(x_{1},\ldots,x_{n}\right)}{\dint_{(t_{n},+\infty)}\varphi_{\xi_{t_{1}}^{s},\ldots,\xi_{t_{n}}^{s}}\left(x_{1},\ldots,x_{n}\right)\mathbb{P}_{\tau}(ds)} \\
\\ =\dfrac{\left(\dfrac{\rho(r)}{B_{t_{n},r}}\right)^{\frac{1}{2}}\exp\left(-\dfrac{1}{2}\,\dfrac{B_{t_{n-1},r}}{B_{t_{n},r}B_{t_{n-1},t_{n}}}\,x_{n}^{2}\right)}{\dint_{(t_{n},+\infty)}\left(\dfrac{\rho(s)}{B_{t_{n},s}}\right)^{\frac{1}{2}}\exp\left(-\frac{1}{2}\,\dfrac{B_{t_{n-1},s}}{B_{t_{n},s}B_{t_{n-1},t_{n}}}\,x_{n}^{2}\right)\mathbb{P}_{\tau}(ds)}\\
\\
  =\left(\dint_{(t_{n},+\infty)}\left(\dfrac{\rho(s)B_{t_{n},r}}{\rho(r)B_{t_{n},s}}\right)^{\frac{1}{2}}\exp\left[\dfrac{1}{2}\,\left(\dfrac{\rho(s)q(r)-\rho(r)q(s)}{B_{t_{n},r}B_{t_{n},s}}\right)x_{n}^{2}\right]\mathbb{P}_{\tau}(ds)\right)^{-1}\\
\\
  =\phi_{\xi_{t_{n}}^{r}}(x_{n}).
\end{array}
\]
Combining all this leads to the formula \eqref{taucondxi}.
\begin{corollary} Assume that Assumptions \eqref{hy1}, \eqref{hyindependent} and \eqref{hymarkov} hold. Then
\begin{itemize}
\item[1.] The conditional law of the random time $\tau$ given $\xi_{t_{1}},\ldots,\xi_{t_{n}}$ is given by

\begin{align}\label{taucondlawmulti}
\mathbb{P}_{\tau\vert\xi_{t_{1}}=x_{1},\ldots,\xi_{t_{n}}=x_{n}}\left(x_{1},\ldots,x_{n},dr\right)= & \dfrac{1}{F(t_{1})}\,\mathbb{I}_{\{x_{1}=0\}}\,\mathbb{I}_{(0,t_{1}]}(r)\,\mathbb{P}_{\tau}(dr) \nonumber\\
\nonumber \\
 & +\sum\limits _{k=1}^{n-1}\psi_{k}(r,x_{k})\,\mathbb{I}_{\{x_{k}\neq0,x_{k+1}=0\}}\,\mathbb{I}_{(t_{k},t_{k+1}]}(r)\,\mathbb{P}_{\tau}(dr) \nonumber\\
\nonumber \\
 & +\phi_{\xi_{t_{n}}^{r}}(x_{n})\,\mathbb{I}_{\{x_{n}\neq0\}}\,\mathbb{I}_{(t_{n},+\infty)}(r)\,\mathbb{P}_{\tau}(dr)
\end{align}
\item[2.]  For any bounded measurable function $\mathfrak{g}$  defined on
		$(0,+\infty)\times \mathbb{R}$, we have

\begin{align}\label{tauximultcond}
\mathbb{E}[\mathfrak{g}(\tau,\xi_{t_{n}})\vert\xi_{t_{1}},\ldots,\xi_{t_{n}}]=& \dint_{(0,t_{1}]}\,  \dfrac{\mathfrak{g}(r,0)}{F(t_{1})}\,\mathbb{P}_{\tau}(dr)\,\mathbb{I}_{\{\xi_{t_{1}}=0\}} \nonumber\\
\nonumber \\
 & +\sum\limits _{k=1}^{n-1}\,\dint_{(t_{k},t_{k+1}]}\,\mathfrak{g}(r,0)\,\psi_{k}(r,\xi_{t_{k}})\,\mathbb{P}_{\tau}(dr)\,\mathbb{I}_{\{\xi_{t_{k}}\neq0,\xi_{t_{k+1}}=0\}} \nonumber \\
\nonumber \\
 & +\dint_{(t_{n},+\infty)}\,\mathfrak{g}(r,\xi_{t_{n}})\phi_{\xi_{t_{n}}^{r}}(\xi_{t_{n}})\,\mathbb{P}_{\tau}(dr)\,\mathbb{I}_{\{\xi_{t_{n}}\neq0\}}
\end{align}

		\end{itemize}
\end{corollary}
	\end{proof}
	
	\begin{proposition}\label{ncordbayesest}
Assume that Assumptions \eqref{hy1}, \eqref{hyindependent} and \eqref{hymarkov} hold. Let $n\in \mathbb{N}^*$ and $0=t_0<t_1<t_2<...<t_n<u$ such that $F(t_1)>0$.  Let $\mathfrak{g}$ be a bounded measurable function defined on
		$(0,+\infty)\times \mathbb{R}$. Then, $\mathbb{P}$-a.s., we have	
		
			\begin{align}
			\mathbb{E}[\mathfrak{g}(\tau,\xi_u)|\xi_{t_1},\xi_{t_2},...,\xi_{t_n}]=& \dint_{(0,t_1]}\frac{\mathfrak{g}(r,0)}{F(t_1)}\mathbb{P}_\tau(dr)\mathbb{I}_{\{\xi_{t_{1}}=0\}} \nonumber \\
		\nonumber \\	& +\sum\limits _{k=1}^{n-1}\dint_{(t_{k},t_{k+1}]}\mathfrak{g}(r,0)\psi_{k}(r,\xi_{t_{k}})\mathbb{P}_{\tau}(dr)\mathbb{I}_{\{\xi_{t_{k}}\neq0,\xi_{t_{k+1}}=0\}} \nonumber  \\
	\nonumber \\		& 	+ \dint_{(t_{n},u]}\mathfrak{g}(r,0)\phi_{\xi_{t_{n}}^{r}}(\xi_{t_{n}})\mathbb{P}_{\tau}(dr)\mathbb{I}_{\{\xi_{t_{n}}\neq0\}} \nonumber  \\
	\nonumber \\	& +\dint_{(u,+\infty)}\,\mathfrak{G}_{t_{n},u}(r,\xi_{t_{n}})\phi_{\xi_{t_{n}}^{r}}(\xi_{t_{n}})\,\mathbb{P}_{\tau}(dr)\,\mathbb{I}_{\{\xi_{t_{n}}\neq0\}} .	\label{equationbeyesextensionxi_u,t_n<u}
				\end{align}	
	\end{proposition}

		\begin{proof} First we split $\mathbb{E}[\mathfrak{g}(\tau,\xi_{u})\vert\xi_{t_{1}},\ldots,\xi_{t_{n}}]$ as follows
	\[
\begin{array}{ll}
\mathbb{E}[\mathfrak{g}(\tau,\xi_{u})\vert\xi_{t_{1}},\ldots,\xi_{t_{n}}]= & \mathbb{E}[\mathfrak{g}(\tau,0)\mathbb{I}_{\{\tau\leq t_{n}\}}\vert\xi_{t_{1}},\ldots,\xi_{t_{n}}]\\
\\
 &+ \mathbb{E}[\mathfrak{g}(\tau,0)\mathbb{I}_{\{t_{n}<\tau\leq u\}}|\xi_{t_{1}},\ldots,\xi_{t_{n}}]\\
\\
 & +\mathbb{E}[\mathfrak{g}(\tau,\xi_{u})\mathbb{I}_{\{u<\tau\}}|\xi_{t_{1}},\ldots,\xi_{t_{n}}]
\end{array}
\]
We obtain from Proposition \eqref{propbayesestimatejusquatn} that
\[
\begin{array}{ll}
\mathbb{E}[\mathfrak{g}(\tau,0)\mathbb{I}_{\{\tau\leq t_{n}\}}\vert\xi_{t_{1}},\ldots,\xi_{t_{n}}]= & \dint_{(0,t_{1}]}\frac{\mathfrak{g}(r,0)}{F(t_{1})}\mathbb{P}_{\tau}(dr)\mathbb{I}_{\{\xi_{t_{1}}=0\}}\\
\\
 & +\sum\limits _{k=1}^{n-1}\dint_{(t_{k},t_{k+1}]}\mathfrak{g}(r,0)\psi_{k}(r,\xi_{t_{k}})\mathbb{P}_{\tau}(dr)\mathbb{I}_{\{\xi_{t_{k}}\neq0,\xi_{t_{k+1}}=0\}},
\end{array}
\]
and
\[
\begin{array}{ll}
\mathbb{E}[\mathfrak{g}(\tau,0)\mathbb{I}_{\{t_{n}<\tau\leq u\}}\vert\xi_{t_{1}},\ldots,\xi_{t_{n}}]= & \dint_{(t_{n},u]}\mathfrak{g}(r,0)\phi_{\xi_{t_{n}}^{r}}(\xi_{t_{n}})\mathbb{P}_{\tau}(dr)\mathbb{I}_{\{\xi_{t_{n}}\neq0\}}.\end{array}
\]
Next we show that 
\begin{equation}
\mathbb{E}[\mathfrak{g}(\tau,\xi_{u})\mathbb{I}_{\{u<\tau\}}|\xi_{t_{1}},\ldots,\xi_{t_{n}}]=\dint_{(u,+\infty)}\,\mathfrak{G}_{t_{n},u}(r,\xi_{t_{n}})\phi_{\xi_{t_{n}}^{r}}(\xi_{t_{n}})\,\mathbb{P}_{\tau}(dr)\,\mathbb{I}_{\{\xi_{t_{n}}\neq0\}} .	\label{condeqtauxi}
\end{equation}
Indeed, since $\xi^{r}$  is a Markov process then for a bounded Borel function $h$ we have  
\begin{align*}
\mathbb{E}[\mathfrak{g}(\tau,\xi_{u})\mathbb{I}_{\left\{ u<\tau\right\} }h\left(\xi_{t_{1}},\ldots,\xi_{t_{n}}\right)] & =\dint_{(u,+\infty)}E[\mathfrak{g}(r,\xi_{u}^{r})h(\xi_{t_{1}}^{r},...,\xi_{t_{n}}^{r})]\mathbb{P}_{\tau}(dr)\\
 & =\dint_{(u,+\infty)}\mathbb{E}[\mathbb{E}[\mathfrak{g}(r,\xi_{u}^{r})h(\xi_{t_{1}}^{r},...,\xi_{t_{n}}^{r})|\xi_{t_{1}}^{r},...,\xi_{t_{n}}^{r}]]\mathbb{P}_{\tau}(dr)\\
 & =\dint_{(u,+\infty)}\mathbb{E}[\mathbb{E}[\mathfrak{g}(r,\xi_{u}^{r})|\xi_{t_{n}}^{r}]h(\xi_{t_{1}}^{r},...,\xi_{t_{n}}^{r})]\mathbb{P}_{\tau}(dr).
\end{align*}
 Using \eqref{Gturx}, for $t_{n}<u<r$, we get 
\[
\begin{array}{lll}
\mathbb{E}[\mathfrak{g}(\tau,\xi_{u})\mathbb{I}_{\{u<\tau\}}h(\xi_{t_{1}},\ldots,\xi_{t_{n}})] & = & \dint_{(u,+\infty)}\mathbb{E}[\mathfrak{G}_{t_{n},u}(r,\xi_{t_{n}}^{r})h(\xi_{t_{1}}^{r},...,\xi_{t_{n}}^{r})]\mathbb{P}_{\tau}(dr)\\
\\
 & = & \mathbb{E}[\mathfrak{G}_{t_{n},u}(\tau,\xi_{t_{n}})\mathbb{I}_{\{u<\tau\}}h(\xi_{t_{1}},\ldots,\xi_{t_{n}})].
\end{array}
\]
It follows from \eqref{tauximultcond}, that $\mathbb{P}$-a.s.

\begin{align}\mathbb{E}[\mathfrak{G}_{t_{n},u}(\tau,\xi_{t_{n}})\mathbb{I}_{\{u<\tau\}}\vert\xi_{t_{1}},\ldots,\xi_{t_{n}}]= & \dint_{(u,+\infty)}\,\mathfrak{G}_{t_{n},u}(r,\xi_{t_{n}})\,\phi_{\xi_{t_{n}}^{r}}(\xi_{t_{n}})\,\mathbb{P}_{\tau}(dr)\,\mathbb{I}_{\{\xi_{t_{n}}\neq0\}}.\end{align}
This induces that
\[
\begin{array}{l}
\mathbb{E}[\mathfrak{g}(\tau,\xi_{u})\mathbb{I}_{\{u<\tau\}}h(\xi_{t_{1}},\ldots,\xi_{t_{n}})]=\\
\\
\mathbb{E}\left(\dint_{(u,+\infty)}\,\mathfrak{G}_{t_{n},u}(r,\xi_{t_{n}})\phi_{\xi_{t_{n}}^{r}}(\xi_{t_{n}})\,\mathbb{P}_{\tau}(dr)\,\mathbb{I}_{\{\xi_{t_{n}}\neq0\}}h(\xi_{t_{1}},\ldots,\xi_{t_{n}})\right).
\end{array}
\]
Hence the formula \eqref{condeqtauxi} is proved and then the proof of the proposition is completed.
		\end{proof}
By combining the fact that $\{\xi_{t_{i}}=0\}\subset\{\xi_{t_{j}}=0\}$ for $j\geq i$, $\dint_{(t_{k},t_{k+1}]}\psi_{k}(r,\xi_{t_{k}})\mathbb{P}_{\tau}(dr)=1$ for $k=1,\ldots,n-1$ and Corollary \eqref{corconditionaldistribution} we obtain 
	\begin{corollary}\label{Markovxi}
Assume that Assumptions \eqref{hy1}, \eqref{hyindependent} and \eqref{hymarkov} hold. Let $n\in \mathbb{N}^*$ and $0=t_0<t_1<t_2<...<t_n<u$.  Let $g$ be a bounded measurable function defined on
		$\mathbb{R}$. Then, $\mathbb{P}$-a.s., we have	
		
			\begin{align}
			\mathbb{E}[g(\xi_u)|\xi_{t_1},\xi_{t_2},...,\xi_{t_n}]=& \,g(0) \left( \mathbb{I}_{\{\xi_{t_{n}}=0\}} 	+ \dint_{(t_{n},u]}\,\phi_{\xi_{t_{n}}^{r}}(\xi_{t_{n}})\mathbb{P}_{\tau}(dr)\mathbb{I}_{\{\xi_{t_{n}}\neq0\}} \right)\nonumber  \\
	\nonumber \\	& +\dint_{(u,+\infty)}\,K_{t_{n},u}(r,\xi_{t_{n}})\phi_{\xi_{t_{n}}^{r}}(\xi_{t_{n}})\,\mathbb{P}_{\tau}(dr)\,\mathbb{I}_{\{\xi_{t_{n}}\neq0\}}	\nonumber  \\
	\nonumber \\ = &\,\mathbb{E}[g(\xi_u)|\xi_{t_n}]. \nonumber
				\end{align}	
	\end{corollary}
	Summarizing the above discussion we are able now to state the claim about the markov property of $\xi$ with respect to its natural filtration. 
	\begin{theorem}\label{thmxitaumarkov}
Under  Assumptions \eqref{hy1}, \eqref{hyindependent} and \eqref{hymarkov}, the Gaussian bridge $\xi$ of random length $\tau$  is an $\mathbb{F}^{\xi}$-Markov process.
	\end{theorem} 
	\begin{proof}
	First, we would like to mention that since $\xi_{0}=0$ almost surely then it is easy to see that 
$$\mathbb{E}[g(\xi_u)\vert \mathcal{F}^{\xi}_{0} ]=\mathbb{E}[g(\xi_u)\vert \xi_{0}].$$	
	The markov property of $\xi$ is then a consequence of Corollary \eqref{Markovxi} and Theorem 1.3 in Blumenthal and Getoor \cite{BG}.
	\end{proof}
	\begin{remark}
		It is not hard to see that the Markov property can be extended to the completed filtration  $\mathbb{F}^{\xi,c}$.
	\end{remark}
We close this section with some extensions of Propositions \eqref{propbayesestimatejusquatn} and \eqref{ncordbayesest} based on the observation of $\xi$ up to time $t$. We denote by $\mathcal{F}^{\xi}_{t,+\infty} := \sigma(\xi_s, t \leq s \leq +\infty) \vee \mathcal{N}_P$ the $\sigma$-algebra generated by the future of $\xi$ at time $t$. We have the following result.
	\begin{proposition}\label{thmbayesestimatemarkovtau}
		Assume that Assumptions \eqref{hy1}, \eqref{hyindependent} and \eqref{hymarkov} hold. Let $t>0$ and $g:\mathbb{R}_{+}\longrightarrow \mathbb{R} $ be a Borel function such that
		$\mathbb{E}[|g(\tau)|]<+\infty$. Then 
		$$ \mathbb{E}[g(\tau)\vert  \mathcal{F}_{t}^{\xi,c}]=g(\tau\wedge t)\mathbb{I}_{\{\xi_t=0\}}+\dint_{(t,+\infty)}g(r)\phi_{\xi_{t}^r}(\xi_{t})\,\mathbb{P}_{\tau}(dr)\,\mathbb{I}_{\{\xi_t\neq0 \}},~~\mathbb{P}-a.s. $$
	\end{proposition}
	\begin{proof} 
		Obviously, we have
		$$\mathbb{E}[g(\tau)| \mathcal{F}_{t}^{\xi,c}]=\mathbb{E}[g(\tau\wedge t)\mathbb{I}_{\{\tau\leqslant t\}}|\mathcal{F}_{t}^{\xi,c}]+\mathbb{E}[g(\tau\vee t)\mathbb{I}_{\{ t<\tau \}}|\mathcal{F}_{t}^{\xi,c}].$$
Now since  $g(\tau\wedge t)\mathbb{I}_{\{\tau\leqslant t\}}$ is $ \mathcal{F}^{\xi,c}_{t}$-measurable then, $\mathbb{P}$-a.s, one has
\[
\begin{array}{lll}
\mathbb{E}[g(\tau\wedge t)\mathbb{I}_{\{\tau\leqslant t\}}\vert\mathcal{F}_{t}^{\xi,c}] & = & g(\tau\wedge t)\mathbb{I}_{\{\tau\leqslant t\}}\\
\\
 & = & g(\tau\wedge t)\mathbb{I}_{\{\xi_{t}=0\}}.
\end{array}
\]
On the other hand due to the fact that $g(\tau\vee t)\mathbb{I}_{\{t< \tau \}}$ is  $ \mathcal{F}^{\xi}_{t,+\infty}$-measurable and $\xi$ is a Markov process with respect to its completed natural filtration we obtain
		$$\mathbb{E}[g(\tau\vee t)\mathbb{I}_{\{ t<\tau \}}\vert \mathcal{F}_{t}^{\xi,c}]=\mathbb{E}[g(\tau\vee t)\mathbb{I}_{\{ t<\tau \}}|\xi_{t}],~~\mathbb{P}-a.s.$$
The result is deduced from \eqref{equationbeyesestimate}. 
	\end{proof}
		Here it is possible to see how the Bayesian estimate given above provides
	a better knowledge on the stopping time $\tau$ through the observation of $\xi$ at time $t$.
	Next we give an extension of the Bayes estimates of $\tau$ which is a consequence of the Markov property of $\xi$ and Proposition \eqref{propextensionxiu}.
	\begin{proposition}\label{thmextensionxiumarkov}
		Assume that Assumptions \eqref{hy1}, \eqref{hyindependent} and \eqref{hymarkov} hold. Let $0 < t< u$ and $g$ be a bounded measurable function defined on
		$(0,+\infty)\times \mathbb{R}$. Then, $\mathbb{P}$-a.s, we have
		\[
		\begin{array}{lll}
		(i) \qquad  \mathbb{E}[g(\tau,\xi_{t})| \mathcal{F}_{t}^{\xi,c}]& = &g(\tau\wedge t,0)\mathbb{I}_{\{\xi_t=0\}}+\dint_{(t,+\infty)}g(r,\xi_{t})\phi_{t}(r,\xi_{t})\mathbb{P}_{\tau}(dr)\mathbb{I}_{\{\xi_t \neq 0\}}. 
				\\ 
				\\
				\\
		(ii) \qquad \mathbb{E}[g(\tau,\xi_{u})|\mathcal{F}_{t}^{\xi,c}] &= &g(\tau\wedge t,0)\mathbb{I}_{\{\xi_t=0\}}+\dint_{(t,u]}g(r,0)\phi_{t}(r,\xi_{t})\mathbb{P}_{\tau}(dr)\mathbb{I}_{\{\xi_t \neq 0\}}\\ \\
			& & +\dint_{(u,+\infty)}\int_{\mathbb{R}} g(r,y)p\left( \dfrac{B_{u,r}B_{t,u}}{B_{t,r}},y,\frac{B_{u,r}}{B_{t,r}}\xi_t\right) dy \phi_{t}(r,\xi_{t})\mathbb{P}_{\tau}(dr)\mathbb{I}_{\{\xi_t \neq 0\}}. 
			\end{array}
		\]
	\end{proposition}
	\begin{remark} 
		The process
		$\xi$ cannot be an homogeneous $\mathbb{F}^{\xi}$-Markov process. Indeed, Proposition \eqref{thmextensionxiumarkov} enables us to see that, for $\Gamma \in \mathcal{B}(\mathbb{R})$ and $t<u$, we have $\mathbb{P}$-a.s.,
		\begin{multline}
		 \mathbb{P}(\xi_{u}\in \Gamma| \mathcal{F}_{t}^{\xi})=\mathbb{I}_{\{0 \in \Gamma \}}\mathbb{I}_{\{\xi_t=0\}}+\mathbb{I}_{\{0 \in \Gamma \}}\int_{(t,u]}\phi_{t}(r,\xi_{t})\mathbb{P}_{\tau}(dr)\mathbb{I}_{\{\xi_t \neq 0\}}\\
		 +\int_{(u,+\infty)}\int_{\Gamma}p\left( \frac{B_{u,r}B_{t,u}}{B_{t,r}},y,\frac{B_{u,r}}{B_{t,r}}\xi_t\right) dy \phi_{t}(r,\xi_{t})\mathbb{P}_{\tau}(dr)\mathbb{I}_{\{\xi_t\neq 0\}} ,\label{equationnonhomogenemarkov}	
		\end{multline}
	 which is clear that it doesn't depend only on $u - t$.
	\end{remark}

\begin{center}
	\section{Markov property with respect to $\mathbb{F}^{\xi,c}_+$}	\label{sectionmarkovpropertyfiltrationcontinue}
\end{center}

We have established, in the previous section, the Markov property of $\xi$ with respect to its completed natural filtration $\mathbb{F}^{\xi,c}$. In this section we are interested in the the Markov property  of $\xi$ with respect to $\mathbb{F}^{\xi,c}_{+}$. It has an interesting consequence which is none other than the filtration $\mathbb{F}^{\xi,c}$ satisfies the usual conditions of completeness and right-continuity.\\
Assume that the function $q$ defined in \eqref{covrep} satisfies the following condition:
\begin{hy}\label{hypconv}
	\begin{equation}
	\underset{t>0}{\inf}\,\,q(t)=\alpha>0.\label{equationhymarkovxi+}
	\end{equation}
\end{hy}
\begin{remark}\label{roqconv}
As a consequence of the above hypothesis, for any decreasing sequence of strictly positive real numbers $(t_{n})_{n\in \mathbb{N}}$ converging to $0$,  we have
\[
\underset{n\rightarrow+\infty}{\lim}\,\dfrac{\rho(t_{n})}{q(t_{n})}=0.
\]
\end{remark}
\begin{theorem}\label{thmmarkovpropertyfiltrationcontinue} 
Assume that Assumptions \eqref{hy1}, \eqref{hyindependent}, \eqref{hymarkov} and \eqref{roqconv} hold. Then the  process $\xi$ is a Markov process with respect to $\mathbb{F}^{\xi,c}_{+}$.
	\end{theorem}
\begin{proof}
It is sufficient to prove that for any $0\leq t<u$ and any function bounded continuous $g$ we have
\begin{equation}
	\mathbb{E}[g(\xi_{u})|\mathcal{F}^{\xi,c}_{t+}]=\mathbb{E}[g(\xi_{u})|\xi_{t}],~~ \mathbb{P}-a.s.\label{equationmarkovrct>0}
	\end{equation}
	  Let $(t_{n})_{n\in \mathbb{N}}$
	be a decreasing sequence of strictly positive real numbers converging to $t$: that is $0 \leq t <...< t_{n+1} < t_{n} < u $, $t_{n} \searrow t$ as $n \longrightarrow +\infty $. Since $g$ is bounded and $\mathcal{F}_{t+}^{\xi,c}=\underset{n}{\cap}\mathcal{F}_{t_{n}}^{\xi,c}$ then,  $\mathbb{P}$-a.s., we have
	\begin{equation}
	\mathbb{E}[g(\xi_{u})|\mathcal{F}^{\xi,c}_{t+}]=\lim\limits_{n\longmapsto +\infty}\mathbb{E}[g(\xi_{u})|\mathcal{F}^{\xi,c}_{t_n}].
	\end{equation}
	It follows from the Markov property of $\xi$ with respect to $\mathbb{F}^{\xi,c}$ that 	
 \begin{equation}
	\mathbb{E}[g(\xi_{u})|\mathcal{F}^{\xi,c}_{t+}]=\lim\limits_{n\longmapsto +\infty}\mathbb{E}[g(\xi_{u})|{\xi}_{t_n}],~~\mathbb{P}-a.s.
	\end{equation}
It remains to prove that
	\begin{equation}
	\lim\limits_{n \longrightarrow +\infty} \mathbb{E}[g(\xi_{u})|\xi_{t_{n}}]=\mathbb{E}[g(\xi_{u})|\xi_{t}],~~\mathbb{P}-a.s.\label{equationcontdemarkovrc}
	\end{equation}
The proof is splitted into two parts. In the first one we show the statment \eqref{equationmarkovrct>0}  for $t > 0$, while in the second part we consider the case $t = 0$.

Let $t>0$. We begin by noticing that from \eqref{promarkovonepar}, $\mathbb{P}$-a.s., we have
	
\begin{align}
			\mathbb{E}[g(\xi_u)|\xi_{t_{n}}]=& g(0) \left( \mathbb{I}_{\{\xi_{t_{n}}=0\}} 	+ \dint_{(t_{n},u]}\,\phi_{\xi_{t_{n}}^{r}}(\xi_{t_{n}})\mathbb{P}_{\tau}(dr)\mathbb{I}_{\{\xi_{t_{n}}\neq0\}} \right)\nonumber  \\
	\nonumber \\	& +\dint_{(u,+\infty)}\,K_{t_{n},u}(r,\xi_{t_{n}})\phi_{\xi_{t_{n}}^{r}}(\xi_{t_{n}})\,\mathbb{P}_{\tau}(dr)\,\mathbb{I}_{\{\xi_{t_{n}}\neq0\}} \nonumber \\ 
	\nonumber \\ & = g(0)\left( \mathbb{I}_{\{\tau \leq  t_{n}\}}+\dint_{(t_{n},u]}\phi_{\xi_{t_{n}}^{r}}(\xi_{t_{n}})\mathbb{P}_{\tau}(dr)\mathbb{I}_{\{t_{n}<\tau \}} \right) \nonumber \\ 
	\nonumber \\ & +\dint_{(u,+\infty)}\,K_{t_{n},u}(r,\xi_{t_{n}})\phi_{\xi_{t_{n}}^{r}}(\xi_{t_{n}})\,\mathbb{P}_{\tau}(dr)\,\mathbb{I}_{\{t_{n}<\tau\}}. \nonumber
	\end{align}
Since $\lim\limits_{n \longrightarrow +\infty}\, \mathbb{I}_{\{t_{n}<\tau \}}=\mathbb{I}_{\{t<\tau \}} $ then assertion \eqref{equationcontdemarkovrc} will be established if we show, $\mathbb{P}$-a.s on $\left\{ t<\tau\right\} $, that
	\begin{equation}
	\lim\limits_{n \longrightarrow +\infty}\dint_{(t_{n},u]}\,\phi_{\xi_{t_{n}}^{r}}(\xi_{t_{n}})\mathbb{P}_{\tau}(dr)=\dint_{(t,u]}\,\phi_{\xi_{t}^{r}}(\xi_{t})\mathbb{P}_{\tau}(dr),\label{equationlimitphitnxi-phitxi}
	\end{equation}
	and
	\begin{equation}
	\lim\limits_{n \longrightarrow +\infty}\dint_{(u,+\infty)}\,K_{t_{n},u}(r,\xi_{t_{n}})\phi_{\xi_{t_{n}}^{r}}(\xi_{t_{n}})\,\mathbb{P}_{\tau}(dr) =\dint_{(u,+\infty)}\,K_{t,u}(r,\xi_{t})\phi_{\xi_{t}^{r}}(\xi_{t})\,\mathbb{P}_{\tau}(dr). \label{equationlimitphitnxiG-phitxiG}
	\end{equation}
We start by proving	 assertion \eqref{equationlimitphitnxi-phitxi}. The integral on the left-hand side of \eqref{equationlimitphitnxi-phitxi} can be rewritten as
	$$ \dint_{(t_{n},u]}\phi_{\xi_{t_{n}}^{r}}(\xi_{t_{n}})\,\mathbb{P}_{\tau}(dr)=\frac{\dint_{(t_{n},u]}\varphi_{\xi_{t_{n}}^{r}}(\xi_{t_{n}})\,\mathbb{P}_{\tau}(dr)}{\dint_{(t_{n},+\infty)}\varphi_{\xi_{t_{n}}^{s}}(\xi_{t_{n}})\mathbb{P}_{\tau}(ds)}. $$
	First let us remark that the function 
\[
(t,r,x)\longrightarrow\varphi_{\xi_{t}^{r}}(x)=p\left(\dfrac{\rho(t)\,B_{t,r}}{\rho(r)},x,0\right)\mathbb{I}_{\left\{ t<r\right\} }
\]
defined on $\ensuremath{(0,+\infty)\times(0,+\infty)\times\mathbb{R}\backslash\{0\}}$
is continuous. Hence, $\mathbb{P}$-a.s on $\left\{ t<\tau\right\} $,
we have 
\begin{equation}
\underset{n\rightarrow+\infty}{\lim}\,\varphi_{\xi_{t_{n}}^{r}}(\xi_{t_{n}})=\varphi_{\xi_{t}^{r}}(\xi_{t}).\label{phixiconv}
\end{equation}
 Now since
\[
\dfrac{\rho(t)\,B_{t,r}}{\rho(r)}=\rho(t)\,q(t)\left[1-\dfrac{\rho(t)}{q(t)}\,\dfrac{q(r)}{\rho(r)}\right],
\]
 and the $\rho$ and $q$ are continuous such that $\ensuremath{q/\rho}$ is decreasing on $\ensuremath{\mathbb{R}_{+}^{*}}$
then for any compact subset $\mathcal{K}$ of $(0,+\infty)\times\mathbb{R}\backslash\{0\}$
it yields 
\[
\underset{(t,x)\in\mathcal{K},r>0}{\sup}\,\varphi_{\xi_{t}^{r}}(x)<+\infty.
\]
 It results, $\mathbb{P}$-a.s on $\left\{ t<\tau\right\} $, that
\begin{equation}
\underset{n\in\mathbb{N},r>t}{\sup}\,\varphi_{\xi_{t_{n}}^{r}}(\xi_{t_{n}})<+\infty.\label{varphiconvergence}
\end{equation}
	We conclude assertion \eqref{equationlimitphitnxi-phitxi} from the Lebesgue dominated convergence theorem.\\
	Now let us prove \eqref{equationlimitphitnxiG-phitxiG}. Recall that the
	function $K_{t_{n},u}(r,\xi_{t_{n}})$ is given by 
	$$K_{t_{n},u}(r,\xi_{t_{n}})=\dint_{\mathbb{R}} g(y)p(\sigma_{t_{n},u}^r,y,\mu_{t_{n},u}^r\,\xi_{t_{n}})\lambda(dy).$$
	We write it with the help of \eqref{conditionalmarkov} as
\begin{equation}
K_{t_{n},u}(r,\xi_{t_{n}})=\dint_{\mathbb{R}} g(y) p\left(\dfrac{B_{t_{n},u}B_{u,r}}{B_{t_{n},r}},y,\dfrac{B_{u,r}}{B_{t_{n},r}}\xi_{t_{n}}\right) \lambda(dy). \label{newK}
\end{equation}
Since $g$ is bounded we deduce that $K_{t_{n},u}(r,\xi_{t_{n}})$ is bounded. Moreover we obtain from the weak convergence of Gaussian measures that  
\[
\underset{n\rightarrow+\infty}{\lim}\, K_{t_{n},u}(r,\xi_{t_{n}})=K_{t,u}(r,\xi_{t}),
\] 
$\mathbb{P}$-a.s on $\left\{ t<\tau\right\} $.
Combining the fact that $K_{t_{n},u}(r,\xi_{t_{n}})$ is bounded, \eqref{phixiconv} and \eqref{varphiconvergence} assertion \eqref{equationlimitphitnxiG-phitxiG} is then derived from the Lebesgue dominated convergence theorem.
	 
	 Next, we investigate the second part of the proof, that is the case $t = 0$. It will be carried out in two steps. In the first one we assume that there exists $\varepsilon > 0$ such that 
	\begin{equation}\label{tauminor}
	\mathbb{P}(\tau > \varepsilon)=1.
	\end{equation} 
 As in the first part, it is sufficient to verify that
	\begin{equation}\label{equationlimitmarkovstep2t=0}
	\lim\limits_{n \longrightarrow +\infty} \mathbb{E}[g(\xi_{u})\vert {\xi}_{t_{n}}]=\mathbb{E}[g(\xi_{u})\vert \xi_{0}],~~\mathbb{P}-a.s.
	\end{equation}
	 Without loss of generality we assume
	$t_{n}< \varepsilon $ for all $n\in \mathbb{N}$. It is easy to see that under condition \eqref{tauminor}, 	$\mathbb{E}[g(\xi_{u})\vert\xi_{t_{n}}]$ takes the form
	\[
	\mathbb{E}[g(\xi_{u})\vert\xi_{t_{n}}]=g(0)\dint_{(\varepsilon,u]}\phi_{\xi_{t_{n}}^{r}}(\xi_{t_{n}})\mathbb{P}_{\tau}(dr)+\dint_{(u,+\infty)}\,K_{t_{n},u}(r,\xi_{t_{n}})\phi_{\xi_{t_{n}}^{r}}(\xi_{t_{n}})\,\mathbb{P}_{\tau}(dr).
	\]
On the other hand we have
\[
\mathbb{E}[g(\xi_{u})\vert\xi_{0}]=\mathbb{E}[g(\xi_{u})]=g(0)F(u)+\int_{(u,+\infty)}\int_{\mathbb{R}}g(x)p\left(\dfrac{\rho(u)\,B_{u,r}}{\rho(r)},x,0\right)\lambda(dx)\mathbb{P}_{\tau}(dr).
\]
Then in order to show \eqref{equationlimitmarkovstep2t=0} it is sufficient to prove, $\mathbb{P}$-a.s, the following
	\begin{equation}
		\lim\limits_{n\longrightarrow +\infty}\;\dint_{(\varepsilon,u]}\phi_{\xi_{t_{n}}^{r}}(\xi_{t_{n}})\mathbb{P}_{\tau}(dr)=F(u),
		\label{conas1}
	\end{equation}
	and 
	\begin{equation}
\lim\limits _{n\longrightarrow+\infty}\,\dint\limits _{(u,+\infty)}\,K_{t_{n},u}(r,\xi_{t_{n}})\phi_{\xi_{t_{n}}^{r}}(\xi_{t_{n}})\,\mathbb{P}_{\tau}(dr)=\dint\limits _{(u,+\infty)}\dint_{\mathbb{R}}\,g(y)\,p\left(\dfrac{\rho(u)\,B_{t,r}}{\rho(r)},x,0\right)\,\lambda(dx)\,\mathbb{P}_{\tau}(dr).	\label{conas2}
	\end{equation}
First let us prove that $\mathbb{P}$-a.s, for all $r>0$,
		$ \lim\limits_{n \longrightarrow +\infty} \phi_{\xi_{t_{n}}^{r}}(\xi_{t_{n}}) = 1$. Using \eqref{Bdef}, \eqref{phimarkdensity} and \eqref{tauminor} the function $\phi_{\xi_{t_{n}}^{r}}$ defined in \eqref{equationphi} can be rewritten as follows:

\[
\begin{array}{lll}
\phi_{\xi_{t_{n}}^{r}}(x)&=&\dfrac{\left(\dfrac{\rho(r)}{\rho(t_{n})\,B_{t_{n},r}}\right)^{\frac{1}{2}} \exp \left(-\dfrac{1}{2}\dfrac{\rho(r)}{\rho(t_{n})\,B_{t_{n},r}}\,x^{2}\right)}{\dint_{(\varepsilon,+\infty)}\left(\dfrac{\rho(s)}{\rho(t_{n})\,B_{t,s}}\right)^{\frac{1}{2}}\exp\left(-\dfrac{1}{2}\dfrac{\rho(s)}{\rho(t_{n})\,B_{t_{n},s}}\,x^{2}\right)\mathbb{P}_{\tau}(ds)},\\ \\
& = & 
\dfrac{\left(1-\dfrac{\rho(t_{n})/q(t_{n})}{\rho(r)/q(r)}\right)^{-\frac{1}{2}}}{\dint_{(\varepsilon,+\infty)}\left(1-\dfrac{\rho(t_{n})/q(t_{n})}{\rho(s)/q(s)}\right)^{-\frac{1}{2}}\exp\left(\dfrac{1}{2}\left(\dfrac{\rho(r)}{\rho(t_{n})\,B_{t_{n},r}}-\dfrac{\rho(s)}{\rho(t_{n})\,B_{t_{n},s}}\right)\,x^{2}\right)\mathbb{P}_{\tau}(ds)},\end{array}
\]
where $x\in \mathbb{R}$ and $r\in (\varepsilon,+\infty)$.
Since $s,r\geq t_{n}$ it follows from \eqref{Bdef} that	
\[
\dfrac{\rho(r)}{\rho(t_{n})\,B_{t_{n},r}}-\dfrac{\rho(s)}{\rho(t_{n})\,B_{t_{n},s}}=
\dfrac{\bigg(\rho(s)/q(s)-\rho(r)/q(r)\bigg)}{q^{2}(t_{n})\bigg(\rho(r)/q(r)-\rho(t_{n})/q(t_{n})\bigg)\bigg(\rho(s)/q(s)-\rho(t_{n})/q(t_{n})\bigg)}.
\]
By a monotony argument it is simple to see that  
\begin{equation}
1\leq \dfrac{1}{\left(1-\dfrac{\rho(t_{n})/q(t_{n})}{\rho(s)/q(s)}\right)}\leq\dfrac{1}{\left(1-\dfrac{\rho(t_{n})/q(t_{n})}{\rho(\varepsilon)/q(\varepsilon)}\right)}\leq\dfrac{1}{\left(1-\dfrac{\rho(t_{1})/q(t_{1})}{\rho(\varepsilon)/q(\varepsilon)}\right)},\,\,
\text{for}\,\, s\in (\varepsilon,+\infty),\label{numest}
\end{equation}
and 
\[
\dfrac{\bigg(\rho(s)/q(s)-\rho(r)/q(r)\bigg)}{\bigg(\rho(s)/q(s)-\rho(t_{n})/q(t_{n})\bigg)}\leq\left\lbrace\begin{array}{lll}
0 &\text{for}  & s\in\left(\varepsilon,r\right)\\
\\
1 & \text{for} & s\in\left(r,+\infty\right).
\end{array}\right.
\]
So for any $s\in (\varepsilon,+\infty)$ we have 
\[
\begin{array}{c}
\left(1-\dfrac{\rho(t_{n})/q(t_{n})}{\rho(s)/q(s)}\right)^{-\frac{1}{2}}\exp\left(\dfrac{1}{2}\left(\dfrac{\rho(r)}{\rho(t_{n})\,B_{t_{n},r}}-\dfrac{\rho(s)}{\rho(t_{n})\,B_{t_{n},s}}\right)\,x^{2}\right)\leq\\
\\
\left(1-\dfrac{\rho(t_{1})/q(t_{1})}{\rho(\varepsilon)/q(\varepsilon)}\right)^{-\frac{1}{2}}\exp\left(\dfrac{1}{2}\,\dfrac{x^{2}}{q^{2}(t_{n})\bigg(\rho(r)/q(r)-\rho(t_{n})/q(t_{n})\bigg)}\,\right).
\end{array}
\]
 Moreover it follows from \eqref{hypconv} and Remark \eqref{roqconv} that
\begin{equation}
\underset{n\in\mathbb{N}}{\sup}\;\;\dfrac{1}{q^{2}(t_{n})\bigg(\rho(r)/q(r)-\rho(t_{n})/q(t_{n})\bigg)}<+\infty.
\end{equation}
Now since $\xi_{t_{n}}$ converges almost surely to $0$ then almost surely, for all $r>0$, the Lebesgue dominated convergence theorem leads to 

\[
\underset{n\rightarrow+\infty}{\lim}\,\dint_{(\varepsilon,+\infty)}\left(1-\dfrac{\rho(t_{n})/q(t_{n})}{\rho(s)/q(s)}\right)^{-\frac{1}{2}}\exp\left(\dfrac{1}{2}\left(\dfrac{\rho(r)}{\rho(t_{n})\,B_{t_{n},r}}-\dfrac{\rho(s)}{\rho(t_{n})\,B_{t_{n},s}}\right)\,\xi_{t_{n}}^{2}\right)\mathbb{P}_{\tau}(ds) =1.
\]
As a consequence we obtain almost surely, for all $r>0$,  $ \lim\limits_{n \longrightarrow +\infty} \phi_{\xi_{t_{n}}^{r}}(\xi_{t_{n}}) = 1$.

Secondly let us prove the uniform boundedness of the sequence $\phi_{\xi_{t_{n}}^{r}}(\xi_{t_{n}})$. That is $\mathbb{P}$-a.s, the function $r\longrightarrow \phi_{\xi_{t_{n}}^{r}}(\xi_{t_{n}})$ is $\mathbb{P}_{\tau}$-a.s. uniformly bounded. Let $s\in (\varepsilon,+\infty)$. First it follows from \eqref{numest} that 

\[
\left(1-\dfrac{\rho(t_{n})/q(t_{n})}{\rho(s)/q(s)}\right)^{-\frac{1}{2}}\geq 1.
\]
 Moreover using  \eqref{hypconv} and a monotony argument we obtain 
\[
\dfrac{\bigg(\rho(s)/q(s)-\rho(r)/q(r)\bigg)}{q^{2}(t_{n})\bigg(\rho(r)/q(r)-\rho(t_{n})/q(t_{n})\bigg)\bigg(\rho(s)/q(s)-\rho(t_{n})/q(t_{n})\bigg)}\geq-\dfrac{1}{\alpha^{2}\bigg(\rho(\varepsilon)/q(\varepsilon)-\rho(t_{1})/q(t_{1})\bigg)}\mathbb{I}_{\left(\varepsilon,r\right)}(s)
\]
By combining this we arrive at
\[
\begin{array}{c}
\left(1-\dfrac{\rho(t_{n})/q(t_{n})}{\rho(r)/q(r)}\right)^{-\frac{1}{2}}\exp\left(\dfrac{1}{2}\left(\dfrac{\rho(r)}{\rho(t_{n})\,B_{t_{n},r}}-\dfrac{\rho(s)}{\rho(t_{n})\,B_{t_{n},s}}\right)\,x^{2}\right)\geq\\
\\
\exp\left(-\dfrac{x^{2}}{\alpha^{2}\bigg(\rho(\varepsilon)/q(\varepsilon)-\rho(t_{1})/q(t_{1})\bigg)}\right)\mathbb{I}_{\left(\varepsilon,r\right)}(s)
+\mathbb{I}_{\left(r,+\infty\right)}(s)\geq \exp\left(-\dfrac{x^{2}}{\alpha^{2}\bigg(\rho(\varepsilon)/q(\varepsilon)-\rho(t_{1})/q(t_{1})\bigg)}\right).
\end{array}
\]
Hence 
\[
\begin{array}{c}
\dint_{(\varepsilon,+\infty)}\left(1-\dfrac{\rho(t_{n})/q(t_{n})}{\rho(s)/q(s)}\right)^{-\frac{1}{2}}\exp\left(\dfrac{1}{2}\left(\dfrac{\rho(r)}{\rho(t_{n})\,B_{t_{n},r}}-\dfrac{\rho(s)}{\rho(t_{n})\,B_{t_{n},s}}\right)\,\xi_{t_{n}}^{2}\right)\mathbb{P}_{\tau}(ds)\geq\\
\\
\exp\left(-\dfrac{\underset{n\in\mathbb{N}}{\sup}\,\,\xi_{t_{n}}^{2}}{\alpha^{2}\bigg(\rho(\varepsilon)/q(\varepsilon)-\rho(t_{1})/q(t_{1})\bigg)}\right).
\end{array}
\]
Now using again \eqref{numest} for $r\in (\varepsilon,+\infty)$ it is easy to conclude that  		
		
\[
\underset{r\in(\varepsilon,+\infty)}{\sup}\underset{n\in\mathbb{N}}{\sup}\,\,\phi_{\xi_{t_{n}}^{r}}(\xi_{t_{n}})\leq \left(1-\dfrac{\rho(t_{1})/q(t_{1})}{\rho(\varepsilon)/q(\varepsilon)}\right)^{-\frac{1}{2}}\exp\left(\dfrac{\underset{n\in\mathbb{N}}{\sup}\,\,\xi_{t_{n}}^{2}}{\alpha^{2}\bigg(\rho(\varepsilon)/q(\varepsilon)-\rho(t_{1})/q(t_{1})\bigg)}\right)<+\infty,
\]		
$\mathbb{P}$ almost surely. As $ \mathbb{P}(\tau > \varepsilon)=1$ we conclude that $\mathbb{P}$-a.s, the function $r\longrightarrow \phi_{\xi_{t_{n}}^{r}}(\xi_{t_{n}})$ is $\mathbb{P}_{\tau}$-a.s uniformly bounded.
Consequently \eqref{conas1} follows from a simple application of the Lebesgue dominated convergence theorem.
In the same way as in the first case ($t>0$) we obtain from the weak convergence of Gaussian measures that  
\[
\begin{array}{ll}
\underset{n\rightarrow+\infty}{\lim}\,K_{t_{n},u}(r,\xi_{t_{n}}) & =\underset{n\rightarrow+\infty}{\lim}\,\dint_{\mathbb{R}}g(x)p\left(\dfrac{B_{t_{n},u}B_{u,r}}{B_{t_{n},r}},x,\dfrac{B_{u,r}}{B_{t_{n},r}}\xi_{t_{n}}\right)\lambda(dx)\\
\\
 & =\dint_{\mathbb{R}}\,g(x)\,p\left(\dfrac{\rho(u)\,B_{t,r}}{\rho(r)},x,0\right)\,\lambda(dx)
\end{array}
\] 
$\mathbb{P}$ almost surely. Indeed we have 
\[
\dfrac{B_{t_{n},u}}{B_{t_{n},r}}=\dfrac{\rho(u)-\left(\rho(t_{n})/q(t_{n})\right)q(u)}{\rho(r)-\left(\rho(t_{n})/q(t_{n})\right)q(r)},
\]
and 
\[
0\leq \dfrac{1}{B_{t_{n},r}}\leq\dfrac{1}{\alpha^{2}\bigg(\rho(r)/q(r)-\rho(t_{n})/q(t_{n})\bigg)}.
\]
Then it follows from Remark \eqref{roqconv} that  
\[
\underset{n\rightarrow+\infty}{\lim}\,\dfrac{B_{t_{n},u}}{B_{t_{n},r}}=\dfrac{\rho(u)}{\rho(r)},
\]
and 
\[
0\leq \underset{n\in \mathbb{N}}{\sup}\, \dfrac{1}{B_{t_{n},r}}<+\infty.
\]
Therefore we obtain $\underset{n\rightarrow+\infty}{\lim}\,\dfrac{B_{u,r}}{B_{t_{n},r}}\xi_{t_{n}}=0$ almost surely and $\underset{n\rightarrow+\infty}{\lim}\,\dfrac{B_{t_{n},u}B_{u,r}}{B_{t_{n},r}}=\dfrac{\rho(u)B_{u,r}}{\rho(r)}$. To get \eqref{conas2} it suffices to note that  
$$\underset{r\in(\varepsilon,+\infty)}{\sup}\underset{n\in\mathbb{N}}{\sup}\,\,K_{t_{n},u}(r,\xi_{t_{n}})\phi_{\xi_{t_{n}}^{r}}(\xi_{t_{n}})<+\infty,$$ 
and moreover apply the Lebesgue dominated convergence theorem.

Finally, we have to consider the general case, that is $\mathbb{P}(\tau >0)=1$. In order to prove the Markov property of $\xi$ with respect to $\mathbb{F}^{\xi,c}_{+}$ at $t = 0$ it is sufficient to show that $\mathcal{F}_{0+}^{\xi,c}$ is $\mathbb{P}$-trivial. This amounts to prove that $\mathcal{F}_{0+}^{\xi}$ is $\mathbb{P}$-trivial since $\mathcal{F}_{0+}^{\xi,c}=\mathcal{F}_{0+}^{\xi} \vee \mathcal{N}_{P}$. For so doing let $\varepsilon > 0$ be fixed and consider the stopping time $\tau_{\varepsilon}=\tau \vee \varepsilon$. We define the process $\xi_{t}^{\tau_{\varepsilon}}$ by $$\left\lbrace\xi_{t}^{\tau_{\varepsilon}};\,t\geq 0\right\rbrace:=\left\lbrace\xi_{t}^{r} \vert_{r=\tau \vee \varepsilon};\,t\geq 0\right\rbrace.$$  
The first remark is that the sets $(\tau_{\varepsilon}>\varepsilon)=(\tau>\varepsilon)$ are equal  and therefore the following equality of processes holds $$\xi_{\cdot}^{\tau_{\varepsilon}}\mathbb{I}_{(\tau>\varepsilon)}=\xi_{\cdot}\;\mathbb{I}_{(\tau>\varepsilon)}.$$
Then for each $A\in \mathcal{F}_{0+}^{\xi}$ there exists $B\in \mathcal{F}_{0+}^{\xi^{\tau_{\varepsilon}}}$ such that  
$$A\cap(\tau>\varepsilon)=B\cap (\tau>\varepsilon).$$
As $\mathbb{P}(\tau_{\varepsilon}>\varepsilon/2)=1$ then  according to the previous case we have that $\mathcal{F}_{0+}^{\xi^{\tau_{\varepsilon}}}$ is $\mathbb{P}$-trivial. That is $\mathbb{P}(B)=0$ or $1$. Consequently we obtain 
$$\mathbb{P}(A\cap(\tau>\varepsilon))=0\text{\,\,or\,\,}\mathbb{P}(A\cap(\tau>\varepsilon))=\mathbb{P}(\tau>\varepsilon).$$
Now if $\mathbb{P}(A)>0$, then there exists $\varepsilon > 0$ such that $\mathbb{P}(A\cap \{\tau> \varepsilon\})>0$. Therefore for all $0<\varepsilon'\leq \varepsilon$ we have 
$$\mathbb{P}(A\cap(\tau>\varepsilon'))=\mathbb{P}(\tau>\varepsilon').$$ 
Passing to the limit as $\varepsilon'$ goes to $ 0$ yields $\mathbb{P}(A\cap(\tau>0))=\mathbb{P}(\tau>0)=1$. It follows that $\mathbb{P}(A)=1$, which ends the proof.
\end{proof}	
\begin{corollary}\label{corFpsatisfiesusualcondition}
Under the Assumptions \eqref{hy1}, \eqref{hyindependent}, \eqref{hymarkov} and \eqref{roqconv}, the  filtration $\mathbb{F}^{\xi,c}$ satisfies the usual conditions of rightcontinuity and completeness.
\end{corollary}
\begin{proof}
	See, e.g., [\cite{BG}, Ch. I, Theorem (8.12)]
\end{proof}

\end{document}